\documentclass[a4paper,11pt,reqno]{amsart}
\usepackage{amsxtra}
\usepackage{color}
\usepackage{amsopn}
\usepackage{amsmath,amsthm,amssymb,arydshln}
\usepackage{mathrsfs,mathtools}
\usepackage{stmaryrd}
\usepackage{hyperref}
\usepackage{enumerate}
\usepackage{xcolor}
\usepackage{float}
\usepackage{pifont}
\usepackage{adjustbox}

\newtheorem{theorem}{Theorem}[section]

\newtheorem{proposition}[theorem]{Proposition}

\newtheorem*{theorem*}{Theorem}

\theoremstyle{definition}
\newtheorem{definition}[theorem]{Definition}

\newtheorem{remark}[theorem]{Remark}

\makeatletter
\@namedef{subjclassname@2020}{%
  \textup{2020} Mathematics Subject Classification}
\makeatother

\newcommand{\R}{{\mathbb R}}

\newcommand{\Z}{\mathbb Z}

\newcommand{\beq}{\begin{equation}}
\newcommand{\eeq}{\end{equation}}

\newcommand{\f}{\varphi}

\renewcommand{\o}{\omega}

\newcommand{\talpha}{\tilde{\alpha}}

\newcommand{\psip}{\psi}

\newcommand{\SU}{{\mathrm{SU}}}

\newcommand{\GL}{{\mathrm {GL}}}

\newcommand{\G}{{\mathrm G}}

\newcommand{\SL}{{\mathrm {SL}}}

\newcommand{\W}{\wedge}

\DeclareMathOperator\tr{tr}

\DeclareMathOperator\End{End}
\DeclareMathOperator\Aut{Aut}

\DeclareMathOperator\ad{ad}

\DeclareMathOperator{\Der}{Der}

\newcommand{\frg}{\mathfrak{g}}

\newcommand{\frn}{\mathfrak{n}}

\newcommand{\frq}{\mathfrak{q}}
\newcommand{\frs}{\mathfrak{s}}

\renewcommand{\gg}{\mathfrak{g}}

\newcommand{\gr}{\mathfrak{r}}

\newcommand{\gsl}{\mathfrak{sl}}

\newcommand{\st}{\ |\ }

\newcommand{\sst}{\scriptscriptstyle}

\newcommand{\la}{\langle}
\newcommand{\ra}{\rangle}

\numberwithin{equation}{section}

\title[Exact G$_2$-structures on compact quotients of Lie groups]{Exact G$_{\mathbf2}$-structures on  compact quotients of Lie groups}

\author{Anna Fino} 
\address{Dipartimento di Matematica ``G.~Peano'' \\ Universit\`a degli Studi di Torino\\
Via Carlo Alberto 10\\
10123 Torino\\ Italy and Department of Mathematics and Statistics, Florida International University, Miami, FL 33199, USA}
\email{annamaria.fino@unito.it, afino@fiu.edu}

\author{Luc\'ia  Mart\'in-Merch\'an}
\address{
Facultad de Ciencias, Universidad de Malaga, Bulevar Louis Pasteur, 31, 29010, Malaga, Spain}
\email{lmmerchan@uma.es}

\author{Alberto Raffero}
\address{Dipartimento di Matematica ``G.~Peano'' \\ Universit\`a degli Studi di Torino\\
Via Carlo Alberto 10\\
10123 Torino\\ Italy}
\email{alberto.raffero@unito.it}


\begin{document}
\begin{abstract} 
We show that the compact quotient $\Gamma\backslash\G$ of a seven-dimensional simply connected Lie group $\G$ by a co-compact discrete subgroup $\Gamma\subset\G$ 
does not admit any exact $\G_2$-structure which is induced by a left-invariant one on $\G$.
\end{abstract}
\maketitle

\vspace{-0.35cm}

\section{Introduction}

 A $\G_2$-structure on a $7$-manifold  $M$ 
is a reduction of the structure group of 
its frame bundle from the linear group  ${\GL}(7,\mathbb{R})$ to the  compact exceptional Lie group $\G_2$.  

The existence of a $\G_2$-structure on $M$ is characterized by the existence of a 3-form $\varphi\in\Omega^3(M)$ satisfying a certain nondegeneracy condition.  
This 3-form induces a Riemannian metric $g_\f$ and an orientation on the manifold, and thus a Hodge star operator $*_\f$. 

When $\varphi$ is  closed and co-closed, namely $d\f=0$ and $d*_\f\f=0$, the intrinsic torsion of the $\G_2$-structure vanishes identically, the Riemannian metric $g_\f$ is Ricci-flat, 
and $\mathrm{Hol}(g_\f)\subseteq\G_2$, see \cite{Bryant-1,FerGray}. In this case, the $\G_2$-structure is said to be {\em torsion-free}.
A $\G_2$-structure defined by a 3-form $\f$ satisfying the weaker condition $d\f=0$ is said to be {\em closed}. 
A closed $\G_2$-structure is called {\em exact} if $[\f]=0\in H^3_{\mathrm{dR}}(M)$, namely if $\f = d\alpha$ for some $\alpha\in\Omega^2(M)$.

Currently, many examples of compact manifolds admitting closed $\G_2$-structures are available, see \cite{CHNP, Joy, JoKa, Kov, KoLe} 
for examples admitting holonomy $\G_2$ metrics, \cite{FFKM} for an example obtained resolving the singularities of an orbifold, 
and \cite{Bal,Bry,CoFe,Fer1, Fer2, FR, KaLa} for examples on compact quotients of Lie groups.  
However, it is still not known whether exact $\G_2$-structures may occur on compact 7-manifolds. 
A negative answer to this problem was given in \cite{FFR,FS} in some special cases.  
In \cite{FFR}, M.~Fern\'andez and the first and third named author of this paper proved that there are no compact examples of the form $(\Gamma\backslash\G,\f)$, 
where $\G$ is a simply connected solvable Lie group with $(2,3)$-trivial Lie algebra $\frg$, namely $b_2(\frg)=0=b_3(\frg)$, $\Gamma\subset\G$ is a cocompact discrete subgroup (lattice), 
and $\f$ is an {\em invariant} exact $\G_2$-structure on $\Gamma\backslash\G$, namely it is induced by a left-invariant exact $\G_2$-structure on $\G$. 
In \cite{FS}, Freibert and Salamon showed that the same conclusion holds, more generally, when the Lie algebra of $\G$ admits a codimension-one nilpotent ideal. 

Motivated by these results, in this article we investigate the existence of invariant exact G$_2$-structures on compact quotients of Lie groups,  
without considering any extra assumption on the properties of the group.    
In particular, we  prove the following result. 
\begin{theorem}\label{MainThm}
A potential compact  $7$-manifold $M$ with an exact $\G_2$-structure $\f$ cannot be of the form $M=\Gamma\backslash\G$, 
where $\G$ is a seven-dimensional simply connected Lie group, $\Gamma\subset\G$ is a cocompact discrete subgroup, 
and the exact $\G_2$-structure $\f$ on $M$ is invariant.
\end{theorem}

The proof of this theorem will be divided into two parts: in Section \ref{NonSolSect} we focus on the case when $\G$ is non-solvable, while we investigate the solvable case in Section \ref{SolSect}. 
We shall deal only with Lie groups that are unimodular, as this is a necessary condition for the existence of lattices \cite{Mil}. 

There is a one-to-one correspondence between left-invariant exact $\G_2$-structures on $\G$ and $\G_2$-structures on the Lie algebra $\frg = \mathrm{Lie}(\G)$ that are exact with respect to the 
Chevalley-Eilenberg differential. This allows us to investigate the existence of exact $\G_2$-structures at the Lie algebra level.  
We recall that a 3-form $\f$ on $\frg$ defines a $\G_2$-structure if and only if the symmetric bilinear map
\begin{equation}\label{G2map}
b_\f:\frg\times\frg\rightarrow\Lambda^7\frg^*, \quad b_\f(v,w) = \frac16\,\iota_v\f \W \iota_w\f\W\f, 
\end{equation}
satisfies the condition $\det(b_\varphi)^{1/9}\neq 0 \in\Lambda^7\frg^*$ and the symmetric bilinear form 
\[
g_\varphi\coloneqq \det(b_\varphi)^{-1/9} \, b_\varphi : \frg \times \frg \rightarrow \R
\]
is positive definite, see e.g.~\cite{Hitchin}. In particular, for any choice of orientation on $\frg$, the map  
$b_\f:\frg\times\frg\rightarrow \Lambda^7\frg^*\cong\R$ has to be positive or negative definite.

By \cite{FR}, there are $4$ non-solvable unimodular Lie algebras admitting closed $\G_2$-structures, up to isomorphism. 
Three of these Lie algebras are decomposable, and a direct computation with the aid of the software Maple 21 shows that $b_\f$ is never definite for every exact 3-form $\f = d\alpha$ on each one of them 
(see Proposition \ref{NonSolDec}). 
The remaining Lie algebra is indecomposable, and we show that the corresponding simply connected Lie group does not admit any lattice (see Proposition \ref{NonSolDec}). 
These results prove Theorem \ref{MainThm} in the case when $\G$ is non-solvable. 

We then focus on the solvable case. 
Here, there is a further constraint that has to be taken into account. Indeed, a solvable Lie group admits lattices only if it is strongly unimodular \cite{Gar} (see Section \ref{SolSect} for the definition). 
The proof of Theorem \ref{MainThm} when $\G$ is solvable follows then from Theorem \ref{ThmSolvable}, where we show that a seven-dimensional strongly unimodular solvable Lie algebra $\frg$ 
does not admit any exact $\G_2$-structure. 
To achieve this result, we first observe that every such Lie algebra is a semidirect product $\frg \cong \frs \rtimes_D\R$, for some codimension-one unimodular ideal $\frs$ of $\frg$, 
which must be solvable and non-nilpotent by \cite{FS}, and some derivation $D\in\Der(\frs)$. 
The strongly unimodular condition on $\frg$ is then encoded into the derivation $D$, while the existence of an exact G$_2$-structure on $\frg$ implies the existence of a certain type of SU(3)-structure on $\frs$. 
Using these constraints together with the classification of six-dimensional unimodular solvable non-nilpotent Lie algebras, 
we show that none of these Lie algebras can occur as an ideal of a strongly unimodular solvable Lie algebra admitting exact $\G_2$-structures. 
As in the proof of Proposition \ref{NonSolDec}, the computations are done with the aid of the software Maple 21.


\subsection*{Notation} 

Given an $n$-dimensional Lie algebra $\frg$, its structure equations with respect to a basis  $\mathcal{B}^*=(e^1,\ldots,e^n)$ of $\frg^*$ are specified by the $n$-tuple 
$(de^1,\ldots,de^n)$, where $d$ denotes the Chevalley-Eilenberg differential of $\frg$. 
The basis of $\frg$ with dual basis $\mathcal{B}^*$ is denoted by $(e_1,\ldots,e_n)$.

The shortening $e^{ijk\cdots}$ for the wedge product of covectors $e^i \W e^j \W e^k \W\cdots$ is used throughout the paper.

\section{The non-solvable case}\label{NonSolSect}

In this section, we deal with the case when the simply connected unimodular Lie group $\G$ is non-solvable.   
We claim that, in such a case, there are no compact 7-manifolds of the form $\Gamma \backslash \G$ admitting invariant exact $\G_2$-structures. 

By \cite{FR}, we know that $\G$ admits left-invariant closed $\G_2$-structures if and only if its Lie algebra $\gg$ is isomorphic to one of the following
\[
\begin{split}
\frq_1	&= \left(-e^{23},-2\,e^{12},2\,e^{13},0,-e^{45},\frac{1}{2}\,e^{46}-e^{47},\frac{1}{2}\,e^{47}\right);  \\
\frq_2	&= \left(-e^{23},-2\,e^{12},2\,e^{13},0,-e^{45},-\mu\,e^{46},(1+\mu)\,e^{47}\right), \quad -1<\mu\leq-\frac12; \\ 
\frq_3	&= \left(-e^{23},-2\,e^{12},2\,e^{13},0, -\mu\,e^{45},\frac{\mu}{2}\,e^{46}-e^{47},e^{46}+\frac{\mu}{2}\,e^{47}\right),\quad \mu>0; \\ 
\frq_4	&=  \left(-e^{23},-2\,e^{12},2\,e^{13},-e^{14}-e^{25}-e^{47},e^{15}-e^{34}-e^{57},2\,e^{67},0\right). 	
\end{split}
\]
The first three Lie algebras appearing in the previous list decompose as $\frq \cong \gsl(2,\R) \oplus \gr$, where $\gr$ is a four-dimensional unimodular centerless solvable Lie algebra, 
while the Lie algebra $\frq_4$ is indecomposable. 

The proof of our claim follows from the next two propositions. 

\begin{proposition}\label{NonSolDec} 
A seven-dimensional unimodular non-solvable Lie algebra $\frg$ does not admit any exact $\G_2$-structure if it is decomposable. 
\end{proposition}
\begin{proof}
By \cite{FR}, $\frg$ is isomorphic to one of $\frq_1,\frq_2,\frq_3$. 
For each one of these Lie algebras, we consider the generic 2-form $\alpha = \sum_{1\leq i < j \leq 7} a_{ij} e^{ij}$, where $a_{ij}\in\R$, and we use the structure equations to compute 
the expression of the generic exact 3-form $\f = d\alpha$. In detail, we obtain 
\begin{enumerate}[$\bullet$]
\item Lie algebra $\frq_1$
\[
\begin{split}
\f	&= -2 a_{24} \,e^{124} -2 a_{25} \,e^{125} -2 a_{26} \,e^{126} -2 a_{27} \,e^{127} +2 a_{34}\,e^{134} +2 a_{35}\,e^{135} +2 a_{36}\,e^{136}\\  
	&\quad + 2 a_{37}\,e^{137}  + a_{15}\, e^{145} -\frac12  a_{16}\, e^{146} + \left( a_{16} - \frac12 a_{17} \right) e^{147} -a_{14}\,e^{234} -a_{15}\,e^{235} \\
	&\quad -a_{16}\,e^{236} - a_{17}\,e^{237} + a_{25}\,e^{245} -\frac12 a_{26} \, e^{246}  + \left( a_{26} - \frac12 a_{27} \right) e^{247} + a_{35}\,e^{345} \\
	&\quad -\frac12 a_{36} \, e^{346} + \left( a_{36} - \frac12 a_{37} \right) e^{347} -\frac12 a_{56}\, e^{456} + a_{67}\,e^{467} -\left(a_{56}+\frac12 a_{57} \right) e^{457};
\end{split}
\]
\item Lie algebra $\frq_2$
\[
\begin{split}
\f	&=	-2a_{24}\, e^{124} - 2a_{25}\, e^{125}- 2a_{26}\, e^{126}- 2a_{27}\, e^{127} + 2a_{34}\, e^{134} + 2a_{35}\, e^{135}+ 2a_{36}\, e^{136}\\ 
	&\quad	+ 2a_{37}\, e^{137} + a_{15}\, e^{145}+ a_{16}\mu\, e^{146}- a_{17}(1 + \mu)\, e^{147} - a_{14}\, e^{234}- a_{15}\, e^{235}\\
	&\quad	- a_{16}\, e^{236}- a_{17}\, e^{237} + a_{25}\, e^{245}+ a_{26}\mu\, e^{246}- a_{27}(1 + \mu)\, e^{247}+ a_{35}\, e^{345}\\
	&\quad	+ a_{36}\mu\, e^{346}- a_{37}(1 + \mu)\, e^{347}- a_{56}(1 + \mu)\, e^{456}+ a_{57}\mu\, e^{457}+ a_{67}\, e^{467};
\end{split}
\]
\item Lie algebra $\frq_3$
\[
\begin{split}
\f	&=		-2a_{24}\, e^{124} - 2a_{25}\, e^{125} - 2a_{26}\, e^{126} - 2a_{27}\, e^{127} + 2a_{34}\, e^{134}  + 2a_{35}\, e^{135} + 2a_{36}\, e^{136}\\
	&\quad	+ 2a_{37}\, e^{137} + a_{15}\mu\, e^{145} - \left(\frac12a_{16}\mu + a_{17}\right) e^{146}  - \left(\frac12a_{17}\mu - a_{16}\right) e^{147} - a_{14}\, e^{234}\\
	&\quad	- a_{15}\, e^{235} - a_{16}\, e^{236} - a_{17}\, e^{237} + a_{25}\mu\, e^{245} - \left(\frac12a_{26}\mu + a_{27}\right) e^{246} - \left(\frac12a_{27}\mu - a_{26}\right) e^{247}\\ 
	&\quad	+ a_{35}\mu\, e^{345} - \left(\frac12a_{36}\mu + a_{37}\right) e^{346} - \left(\frac12a_{37}\mu - a_{36}\right) e^{347} - \left(\frac12a_{56}\mu - a_{57}\right) e^{456}\\
	&\quad	 - \left(\frac12a_{57}\mu + a_{56}\right) e^{457} + a_{67}\mu\, e^{467}.
\end{split}
\]     
\end{enumerate} 
Now, a direct computation with the aid of the software Maple 21 shows that in each case the bilinear map $b_\f$ defined in \eqref{G2map} satisfies $b_\f(e_i,e_i) = 0$, for $i=5,6,7$. 
Consequently, $\f=d\alpha$ does not define a $\G_2$-structure on $\frg_k$, for $k=1,2,3$. 
\end{proof}

\begin{proposition}\label{NonSolIndec}
Let $\mathrm{Q}_4$ be the simply connected Lie group with Lie algebra $\frq_4$. Then, $\mathrm{Q}_4$ does not admit any lattice. 
\end{proposition}
\begin{proof}
The Lie algebra $\frq_4$ is isomorphic to a semi-direct product of the form $\gsl(2,\R) \ltimes \gr$, 
where the semisimple part is spanned by $e_1,e_2,e_3$, and the four-dimensional radical $\gr = \R\ltimes_D \R^3$ is almost abelian, with 
$\R = \langle e_7\rangle$, $\R^3=\langle e_4,e_5,e_6\rangle$ and 
\[
D \coloneqq \ad_{e_7}|_{\R^3} = \begin{pmatrix} -1 & 0 & 0 \\ 0 & -1 & 0 \\ 0 & 0 & 2 \end{pmatrix}. 
\] 
In particular, the radical of $\mathrm{Q}_4$ is the almost abelian Lie group $\R\ltimes_\mu \R^3$, 
where the one-parameter group $\mu : \R \rightarrow  \Aut(\R^3) $ is defined by the condition $d\mu(t) = \exp(tD)$. 

Now, by \cite[Prop.~1.3]{Wu},				
if $\mathrm{Q}_4$ has a lattice, then also its radical does. By \cite{Boc}, in such a case there must be some $t' \in\R\smallsetminus\{0\}$ such that 
\[
\exp(t' D) = \begin{pmatrix} e^{-t'} & 0 & 0 \\ 0 & e^{-t'} & 0 \\ 0 & 0 & e^{2t'} \end{pmatrix}
\]
is conjugate to a matrix in $\SL(3,\Z)$. This is not possible by  \cite[Lemma~B.4]{Boc}. 
\end{proof}

\section{The solvable case}\label{SolSect}
We now assume that $\G$ is solvable. 
As shown in \cite{Gar}, a simply connected solvable Lie group admitting lattices must be {\em strongly unimodular} according to the following. 
\begin{definition}\label{suDef}
Let $\mathrm{G}$ be a simply connected solvable Lie group with Lie algebra $\frg$ and nilradical $\frn$.  
For each positive integer $i\geq 1$, let $\frn^{\sst i} \coloneqq [\frn,\frn^{\sst i-1}]$ denote the $i^{th}$ term in the descending central series of $\frn$, where $\frn^{\sst0} = \frn$. 
The Lie algebra $\frg$ is {\em strongly unimodular} if for all $X\in\frg$ the restriction of $\ad_{\sst X}$ to each space $\frn^{\sst i}/\frn^{\sst i+1}$ is traceless. 
In this case, the Lie group $\mathrm{G}$ is said to be {\em strongly unimodular}.
\end{definition}

As the name suggests, strongly unimodular Lie groups are unimodular, but the converse does not hold in general, see for instance \cite{FFR}. 

\smallskip

The proof of Theorem \ref{MainThm} in the case when $\G$ is solvable follows from the next result. 

\begin{theorem}\label{ThmSolvable}
A seven-dimensional  strongly unimodular solvable Lie algebra $\frg$  does not  admit any exact  $\G_2$-structure.
\end{theorem}

Before describing the strategy of the proof, we discuss some preliminary results. 
Let $\frg$ be a unimodular solvable Lie algebra endowed with a $\G_2$-structure $\f$. 
Then, it has a codimension-one unimodular ideal $\frs$, and we can consider the $g_\f$-orthogonal decomposition $\frg = \frs \oplus\R$, 
where $\R$ denotes the orthogonal complement of $\frs$. As a Lie algebra, $\frg$ is then a semidirect product of the form $\frg = \frs \rtimes_D\R$, for some derivation $D$ of $\frs$. 
The $\G_2$-structure $\f$ on $\frg$ can be written as follows
\[
\f = \omega\W \eta + \psip,
\]
where $\eta\coloneqq z^\flat$ is the metric dual of a unit vector $z\in\R$, and the pair $(\omega,\psip)$ defines an SU(3)-structure on $\frs$. In detail (see also \cite{Hitchin}): 
\begin{enumerate}[a)]
\item $\omega\in\Lambda^2\frs^*$ is a non-degenerate 2-form, i.e., $\omega^3 = \omega\W\omega\W\omega\neq0$;
\item $\psip\in\Lambda^3\frs^*$ is a {\em negative stable} 3-form, namely a stable 3-form whose associated quartic polynomial satisfies $\lambda(\psip)<0$. 
Here, $\lambda(\psip) \coloneqq \frac16\mathrm{tr}(K_{\psip}^2)$, 
where $K_{\psip}\in\End(\frs)$ is defined as follows. 
Let $A:\Lambda^5\frs^*\rightarrow \frs\otimes \Lambda^6\frs^*$ be the isomorphism induced by the wedge product $\W : \Lambda^5\frs^*\otimes \frs^* \rightarrow \Lambda^6\frs^*$, then 
$K_{\psip}(v)\otimes \omega^3 = A(\iota_v\psip\W\psip)$, for all $v\in\frs$. 
In particular, $K_{\psip}^2 = \lambda(\psip)\mathrm{Id}_\frs$, so that $(\omega,\psip)$ determines an almost complex structure 
\begin{equation}\label{Jpsi}
J:\frs\rightarrow\frs,\quad J = \frac{1}{\sqrt{-\lambda(\psip)}} \, K_{\psip};
\end{equation}
\item $\psip$ is primitive with respect to $\omega$, i.e., $\psip\W\omega=0$. This is equivalent to $\omega$ being of type $(1,1)$ with respect to $J$, namely $\omega(J\cdot,J\cdot)=\omega$;
\item the symmetric bilinear form $g \coloneqq \omega(\cdot,J\cdot)$ is positive definite.
\end{enumerate}
\begin{remark}
More generally, given a stable 3-form $\psip$ on $\frs$, one can define the endomorphism $K_\psip$ by choosing {\em any} volume form $\Omega$ on $\frs$ in place of $\omega^3$,  
and the sign of $\lambda(\psip)$ does not depend on this choice.  
Moreover, if $\lambda(\psip)<0$, the almost complex structure $J$ depends only on $\psip$ and on the orientation of $\frs$. 
Changing the orientation, one obtains the almost complex structure $-J$. 
Finally, we recall that $\psip$ is a negative stable $3$-form if and only if the contraction $\iota_v\psi$ has rank four for every non-zero vector $v\in\frs$. 
\end{remark}

Using the definition of the Chevalley-Eilenberg differential $d$ of $\frg$, we also see that $d\eta=0$. 
Indeed, for every $x,y\in\frg$ we have
\[
d\eta(x,y) = -\eta([x,y]) = -g_\f(z,[x,y]) = 0,
\]
since $[x,y]\in\frs = \langle z\rangle^{\perp_{g_\f}}$.

\smallskip

Assume now that $\f$ is an exact $\G_2$-structure on $\frg = \frs \rtimes_D\R$, namely $\f=d\talpha$ for some $\talpha\in\Lambda^2\frg^*$. 
By \cite{FS}, we know that if $\frg$ is strongly unimodular, then the solvable ideal $\frs$ is not nilpotent. 

We can write $\talpha = \alpha + \beta \wedge \eta$, where $\alpha\in\Lambda^2 \frs^*$ and $\beta \in \frs^*$. 
Then, 
\[
\f =  \hat{d}\alpha +D^*\alpha\wedge \eta+ \hat{d}\beta\wedge \eta = \left(\hat{d}\beta +D^*\alpha \right) \wedge \eta +  \hat{d}\alpha,
\]
where $\hat{d}$ denotes the Chevalley-Eilenberg differential of $\frs$, and the action of $D\in\mathrm{Der}(\frs)$ on $\Lambda^2\frs^*$ is defined as follows 
\[
D^*\alpha(x_1,x_2) = -\alpha(Dx_1,x_2) - \alpha(x_1,Dx_2),
\]
for all $x_1,x_2\in\frs$. From this, we see that $\frs$ has an SU(3)-structure defined by the pair 
\begin{equation}\label{SU3ideal}
\omega \coloneqq \hat{d}\beta + D^*\alpha,\qquad \psip \coloneqq \hat{d}\alpha.
\end{equation} 
In particular, $\psip$ is an exact stable 3-form on $\frs$. 

\smallskip

The previous discussion highlights some necessary conditions imposed by the existence of an exact $\G_2$-structure on a seven-dimensional (strongly) unimodular solvable Lie algebra $\frg = \frs\rtimes_D\R$. 
To show Theorem \ref{ThmSolvable}, we can then proceed as follows.   
The ideal $\frs$ is a six-dimensional unimodular solvable non-nilpotent Lie algebra. 
The Lie algebras satisfying these properties are classified up to isomorphism, so we can investigate each case separately. 
First, we determine which of these Lie algebras do not admit any negative stable exact 3-form, and we rule them out. 
For each one of the remaining Lie algebras, we consider the generic derivation $D\in\Der(\frs)$ and we determine the conditions guaranteeing that the extension 
$\frg = \frs\rtimes_D\R$ is strongly unimodular. 
Then, we investigate whether a generic pair $(\omega,\psip)$ of the form \eqref{SU3ideal} can define an SU(3)-structure on $\frs$. If this is not the case, then we rule $\frs$ out. 
As we will see, none of the six-dimensional unimodular solvable non-nilpotent Lie algebras passes both tests. 
From this, the proof of Theorem \ref{ThmSolvable} follows. 

\begin{remark}
If the derivation $D\in\Der(\frs)$ is not nilpotent, then the nilradical of $\frg =  \frs\rtimes_D\R$ coincides with the nilradical $\frn$ of $\frs$. Otherwise, it is given by $\frn\rtimes_D\R$. 
\end{remark}

The structure equations of all six-dimensional unimodular solvable non-nilpotent Lie algebras can be found in the literature. 
Here, we consider the list given in \cite[Appendix A]{Boc}, where the classification results of various preceding works have been meticulously collected. 
The structure equations of the decomposable unimodular Lie algebras can be determined from the tables A.1, A.3, A.4, A.5, A.6, A.7 in \cite{Boc},  
and they are listed in Table \ref{tabledec} of Appendix \ref{appendix}, where the unimodular Lie algebra $\frg_{4,2}^{-2}$ not appearing in Table A.1 of \cite{Boc} is also included (see \cite{Mub}). 
The structure equations of the unimodular indecomposable Lie algebras are given in the tables A.9 -- A.19 of \cite{Boc}, and we refer the reader to it for the list.  

In what follows, the non-abelian Lie algebras are denoted as in \cite{Boc}, namely we use the symbol $\frg_{n,k}$ to denote the $k^{th}$ Lie algebra of dimension $n$ appearing in the list of  non-isomorphic 
$n$-dimensional solvable Lie algebras. 
Moreover, superscripts like $\frg^{p,q,\ldots}$ denote the values of the real parameters on which a Lie algebra depends. 
Finally, we denote the $n$-dimensional abelian Lie algebra $n\frg_1$ by $\R^n$. 

\smallskip

We will investigate the decomposable and the indecomposable case separately. 

\subsection{The decomposable case}
We begin considering the Lie algebras listed in Table \ref{tabledec} of Appendix \ref{appendix}. 
The next result shows that most of them cannot occur as an ideal of a unimodular solvable Lie algebra admitting exact $\G_2$-structures. 

\begin{proposition}\label{lambdadec}
A six-dimensional unimodular decomposable solvable non-nilpotent Lie algebra $\frs$ admits negative stable exact 3-forms if and only if it is isomorphic to one of the following: 
$\frg_{3,4}^{-1} \oplus \frg_{3,4}^{-1}$, $\frg_{5,30}^{-  4 / 3} \oplus \R$, $\frg_{5,33}^{-1,-1}\oplus \R$, $\frg_{5,35}^{-2,0}\oplus \R$. 
\end{proposition}
\begin{proof}
Let $\frs$ be one of the Lie algebras appearing in Table \ref{tabledec}, and let $(e^1,\ldots,e^6)$ be the basis of $\frs^*$ used to describe the structure equations. 
We consider a generic 2-form $\alpha = \sum_{1\leq i<j\leq6} a_{ij}e^{ij}\in\Lambda^2\frs^*$, where $a_{ij}\in\R$, and we compute its Chevalley-Eilenberg differential $\hat{d}\alpha$ using the structure equations. 
Then, we determine the quartic polynomial $\lambda(\hat{d}\alpha)$ as explained before.  
Notice that we are free to choose the 6-form $e^{123456}$ in place of a generic non-zero element in $\Lambda^6\frs^*$ representing $\omega^3$, 
since the sign of $\lambda(\hat{d}\alpha)$ does not depend on the choice of orientation for $\frs$. 
The Lie algebras for which $\lambda(\hat{d}\alpha)\geq0$ are the following
\[
\renewcommand\arraystretch{1.5}
\begin{array}{rl}
\frg_{3,4}^{-1} \oplus \frg_{3,1}:		&	\lambda(\hat{d}\alpha) = 4 a_{14}^2 a_{24}^2; \\
\frg_{3,4}^{-1} \oplus \frg_{3,5}^0:	&	\lambda(\hat{d}\alpha) = 4\left(a_{14}^2+a_{15}^2\right)\left(a_{24}^2+a_{25}^2\right); \\
\frg_{3,5}^0 \oplus \frg_{3,1}:		&	\lambda(\hat{d}\alpha) = \left(a_{14}^2+a_{24}^2\right)^2; \\
\frg_{3,5}^0 \oplus \frg_{3,5}^0:		&	\lambda(\hat{d}\alpha) = \left( (a_{14}+a_{25})^2 +  (a_{15} - a_{24})^2 \right) \left( (a_{14} - a_{25})^2 +  (a_{15} + a_{24})^2 \right); \\ 
\frg_{5,19}^{p,-2p - 2} \oplus \R:		&	\lambda(\hat{d}\alpha) = 4\,(1+p)^2\, a_{14}^2 a_{16}^2; \\
\frg_{5,23}^{-4}\oplus \R:			&	\lambda(\hat{d}\alpha) = 16\, a_{14}^2 a_{16}^2; \\
\frg_{5,25}^{4,4p}\oplus \R:		&	\lambda(\hat{d}\alpha) = 16\,p^2\, a_{14}^2 a_{16}^2; \\
\frg_{5,28}^{- 3 / 2} \oplus \R:		&	\lambda(\hat{d}\alpha) = a_{14}^2 a_{16}^2.
\end{array}
\]
As for the Lie algebras of Table \ref{tabledec} that are not isomorphic to any of the previous ones nor to one of  
$\frg_{3,4}^{-1} \oplus \frg_{3,4}^{-1}$, $\frg_{5,30}^{-  4 / 3} \oplus \R$, $\frg_{5,33}^{-1,-1}\oplus \R$, $\frg_{5,35}^{-2,0}\oplus \R$, we have $\lambda(\hat{d}\alpha) = 0$. 
On the remaining Lie algebras, there exist exact 3-forms $\hat{d}\alpha$ such that $\lambda(\hat{d}\alpha) < 0$. 
The expression of $\lambda(\hat{d}\alpha)$ for these Lie algebras will be given in the proofs of the next propositions. 
\end{proof}

We are left with the decomposable Lie algebras $\frg_{3,4}^{-1} \oplus \frg_{3,4}^{-1}$, $\frg_{5,30}^{-  4 / 3} \oplus \R$, $\frg_{5,33}^{-1,-1}\oplus \R$, $\frg_{5,35}^{-2,0}\oplus \R$.
We divide the discussion into three propositions, as we use different strategies to rule them out.

\begin{proposition}
The Lie algebras $\frg_{5,30}^{-  4 / 3} \oplus \R$ and $\frg_{5,35}^{-2,0}\oplus \R$ cannot occur as an ideal of a strongly unimodular solvable Lie algebra admitting exact $\G_2$-structures. 
\end{proposition}
\begin{proof}
Let $\frs = \frg_{5,30}^{-  4 / 3} \oplus \R$, and consider the basis $(e^1,\ldots,e^6)$ of $\frs^*$ for which the structure equations are those given in Table \ref{tabledec}, namely
\[
\frg_{5,30}^{-  4 / 3} \oplus \R = \left( - e^{24} - \frac 23 e^{15}, - e^{34} + \frac 13 e^{25}, \frac 43 e^{35}, - e^{45}, 0,0 \right).
\]
Let $\mathcal{B} = (e_1,\ldots,e_6)$ be the basis of $\frs$ with dual basis $(e^1,\ldots,e^6)$. 
The generic derivation $D\in\Der(\frs)$ has the following matrix representation with respect to the basis $\mathcal{B}$ 
\[
D = \left( 
\begin{array}{cccccc}
a_1	&	a_2	&	0	&	a_3	&	a_4	&	0	\\
0	&	a_5	&	a_2	&	a_6	&	-\tfrac13 a_3	&	0	\\
0	&	0	&	2a_5 - a_1	&	0	&	-\tfrac43 a_6	&	0	\\
0	&	0	&	0	&	a_1 - a_5	&	-a_2	&	0	\\
0	&	0	&	0	&	0	&	0	&	0	\\
0	&	0	&	0	&	0	&	a_7	&	a_8	
\end{array}
\right),
\]
where $a_1,\ldots,a_8\in\R$. 

The nilradical of $\frs$ is $\frn = \langle e_1,e_2,e_3,e_4,e_6 \rangle$, and it has the following descending central series
\[
\frn^{\sst0} = \frn,\quad \frn^{\sst1} = \langle e_1,e_2\rangle,\quad  \frn^{\sst2} = \langle e_1\rangle, \quad \frn^{\sst k} =  \{0\},~k\geq3.
\]
From this, we see that the Lie algebra $\frg = \frs \rtimes_D\R$ is strongly unimodular only when $a_1 = a_5 = a_8 = 0$. 

We now consider a generic 2-form  $\alpha = \sum_{1\leq i<j\leq6} \alpha_{ij}e^{ij}$ and a generic 1-form $\beta = \sum_{i=1}^6\beta_k e^k$ on $\frs$, and we compute the forms
\[
\omega =  \hat{d}\beta + D^*\alpha,\qquad \psip = \hat{d}\alpha.
\]
Then, for the values of the parameters $\alpha_{ij}$ and $b_k$ for which  $\omega^3\neq0$ and 
\[
 \lambda(\hat{d}\alpha) = \frac49 \left(4\alpha_{13}^{2} \alpha_{16}^{2}+ 2 \alpha_{16} \alpha_{26} \alpha_{12} \alpha_{13}
									-10 \alpha_{12} \alpha_{16}^{2} \alpha_{23}-4\alpha_{12}^{2} \alpha_{16} \alpha_{36}+\alpha_{12}^{2} \alpha_{26}^{2}\right) < 0, 
\]
we determine the almost complex structure $J$ using the formula \eqref{Jpsi}. 
Notice that the sign of $J$ depends on $\omega^3$ being a positive or negative multiple of the volume form $e^{123456}$. 
Now, a direct computation shows that $\omega(e_i,Je_i) = 0$, for $i=1,2,3$. Therefore, the pair $(\omega,\psi)$ cannot define an SU(3)-structure on $\frs$. 

\smallskip

Similar computations for the Lie algebra 
\[
\frg_{5,35}^{-2,0}\oplus \R = \left(2 e^{14}, - e^{24} - e^{35}, e^{25}- e^{34}, 0,0,0\right)
\]
show that the derivation $D$ must have the following matrix representation
\[
D = \left(
\begin{array}{cccccc}
a_{1} & 0 & 0 & a_{2} & 0 & 0 
\\
 0 & a_{3} & a_{4} & a_{5} & a_{6} & 0 
\\
 0 & -a_{4} & a_{3} & a_{6} & -a_{5} & 0 
\\
 0 & 0 & 0 & 0 & 0 & 0 
\\
 0 & 0 & 0 & 0 & 0 & 0 
\\
 0 & 0 & 0 & a_{7} & a_{8} & -a_{1}-2 a_{3} 
\end{array}\right),
\]
and that whenever
\[
\lambda(\hat{d}\alpha) = 4 \left(\alpha_{26}^{2}+\alpha_{36}^{2}\right) \left(\alpha_{12}^{2}+\alpha_{13}^{2}\right)+16\, \alpha_{23} \alpha_{16} \left(\alpha_{12} \alpha_{36}-\alpha_{13} \alpha_{26}\right)<0,
\]
we have
\[
\omega(e_1,Je_1)\,\omega(e_6,Je_6) =	\frac{1}{\lambda(\hat{d}\alpha)}
								16\, \alpha_{16}^{2} \left(\alpha_{26}^{2}+\alpha_{36}^{2}\right) \left(\alpha_{12}^{2}+\alpha_{13}^{2}\right) \left(a_{1}+2 a_{3}\right)^{2}  \leq 0,
\]
whence the thesis follows. 
\end{proof}

\begin{proposition}\label{3434}
The Lie algebra $\frg_{3,4}^{-1} \oplus \frg_{3,4}^{-1}$ cannot occur as an ideal of a strongly unimodular solvable Lie algebra admitting exact $\G_2$-structures. 
\end{proposition}

\begin{proof}
Let $\frs = \frg_{3,4}^{-1} \oplus \frg_{3,4}^{-1}$, and let $\mathcal{B}^*=(e^1,\ldots,e^6)$ be the basis of $\frs^*$ for which the structure equations are 
\[
 \frg_{3,4}^{-1} \oplus \frg_{3,4}^{-1} = \left(-e^{13}, e^{23},0,-e^{46}, e^{56},0\right). 
\]
Let $\mathcal{B}=(e_1,\ldots,e_6)$ be the basis of $\frs$ with dual basis $\mathcal{B}^*$. 
Then, the nilradical of $\frs$ is the abelian ideal $\frn =\langle e_1,e_2,e_4,e_5\rangle$, and the generic derivation $D\in\Der(\frs)$ for which $\frs\rtimes_D\R$ is strongly unimodular 
has the following matrix representation with respect to $\mathcal{B}$
\[
D = \left(
\begin{array}{cccccc}
a_{1} & 0 & a_{2} & 0 & 0 & 0 
\\
 0 & a_{3} & a_{4} & 0 & 0 & 0 
\\
 0 & 0 & 0 & 0 & 0 & 0 
\\
 0 & 0 & 0 & a_{5} & 0 & a_{6} 
\\
 0 & 0 & 0 & 0 & -a_{5}-a_{3}-a_{1} & a_{8} 
\\
 0 & 0 & 0 & 0 & 0 & 0 
\end{array}\right),
\]
where $a_i\in\R$.

We consider a generic 2-form $\alpha = \sum_{1\leq i<j\leq6} \alpha_{ij}e^{ij}$, a generic 1-form $\beta = \sum_{i=1}^6\beta_k e^k$ on $\frs$, and  
the forms $\omega = \hat{d}\beta + D^*\alpha$ and $\psi = \hat{d}\alpha$. We have
\[
\begin{split}
\omega 	=&	-\alpha_{12}  ( a_1 + a_3 )e^{12}+(- a_1  \alpha_{13} - a_4  \alpha_{12} -\beta_1 )e^{13}-\alpha_{14}  ( a_1 + a_5 )e^{14}+ \alpha_{15}  ( a_5 + a_3 )e^{15} \\[2pt]
		&	+ (- a_1  \alpha_{16} - a_6  \alpha_{14} - a_8  \alpha_{15} )e^{16}+ (a_2  \alpha_{12} - a_3  \alpha_{23} +\beta_2 )e^{23} -\alpha_{24}  ( a_5 + a_3 )e^{24}\\[2pt]
		&	+ \alpha_{25}  ( a_1 + a_5 )e^{25}+(- a_3  \alpha_{26} - a_6  \alpha_{24} - a_8  \alpha_{25} )e^{26}+(- a_2  \alpha_{14} - a_4  \alpha_{24} - a_5  \alpha_{34} )e^{34}\\[2pt]
		&	+( a_1  \alpha_{35} - a_2  \alpha_{15} + a_3  \alpha_{35} - a_4  \alpha_{25} + a_5  \alpha_{35} )e^{35}+(- a_2  \alpha_{16} - a_4  \alpha_{26} - a_6  \alpha_{34} - a_8  \alpha_{35} )e^{36}\\[2pt]
		&	+ \alpha_{45}  ( a_1 + a_3 )e^{45}+(- a_5  \alpha_{46} - a_8  \alpha_{45} -\beta_4 )e^{46}+( a_1  \alpha_{56} + a_3  \alpha_{56} + a_5  \alpha_{56} + a_6  \alpha_{45} +\beta_5 )e^{56},
\end{split} 
\]
and
\[
\begin{split}
\psip 	=& -\alpha_{14}  e^{134}   -\alpha_{15}  e^{135} - \alpha_{16}  e^{136}+ \alpha_{14}  e^{146} - \alpha_{15}  e^{156} +\alpha_{24}  e^{234} + \alpha_{25}  e^{235}   \\
		& + \alpha_{26}  e^{236} + \alpha_{24}  e^{246} -\alpha_{25}  e^{256}  +\alpha_{34} e^{346} -\alpha_{35} e^{356}. 
 \end{split} 
\]
Choosing the volume form $e^{123456}$, we compute 
\[
\lambda(\psip) = 16\, \alpha_{14}\, \alpha_{15}\, \alpha_{24}\, \alpha_{25}.
\]
Assuming that $\lambda(\psip)<0$, we determine the almost complex structure $J$ induced by $\psip$ and the chosen orientation. 
We now show that there exists a nonzero vector $x\in\frs$ such that $\omega(x,Jx)=0$. From this, the thesis follows. 

Let us consider the family of nilpotent derivations $S\in\Der(\frs)$ having the following matrix representation with respect to $\mathcal{B}$:  
\[
S= \left( 
\begin{array}{cccccc}
 0&0& s_1 & 0&  0 & 0\\[2pt]
 0 & 0 &  s_2 & 0 &  0 & 0\\[2pt]
0 &  0 &  0 & 0 & 0 & 0\\[2pt]
0 &  0 &  0 &  0 &0&  s_3\\[2pt]
 0 & 0 &  0 &  0 &  0 &  s_4\\[2pt]
 0 & 0 &  0 &  0 &  0 &  0 
 \end{array}\right),
\]
where $s_i\in\R$, and let  $F \coloneqq \exp (S)\in\mathrm{Aut}(\frs)$. 
From the pair $(\omega,\psip)$, we obtain the pair $(F^*\omega,F^*\psip)$ with associated almost complex structure $J_{F^*\psip} = F^{-1} \circ J \circ F.$

We claim that there exists a choice of the real numbers $s_i$ for which $F^* \omega (e_3,e_6) =0$ and  $J_{F^*\psip} (e_3) \in \la e_3,e_6 \ra$.
This implies that $\omega(x,Jx)=0$ for $x=Fe_3$, as
\[
\omega(Fe_3,JFe_3) = F^*\omega(e_3,J_{F^*\psip}e_3) = 0. 
\]
Comparing 
\[
\begin{split}
F^*\psip 	=&~  \alpha_{25}  e^{235} -\alpha_{25}  e^{256} +\alpha_{24}  e^{234} + \alpha_{24}  e^{246} -\alpha_{15}  e^{135}- \alpha_{15}  e^{156}-\alpha_{14}  e^{134}+ \alpha_{14}  e^{146} \\
		&	- (\alpha_{14}  s_3 + \alpha_{15}  s_4 + \alpha_{16} ) e^{136} + (\alpha_{24}  s_3+\alpha_{25}  s_4+\alpha_{26} )  e^{236}  \\
		& + (\alpha_{14}  s_1+\alpha_{24} s_2+\alpha_{34}) e^{346} -(\alpha_{15}  s_1+ \alpha_{25}  s_2+\alpha_{35}) e^{356},
\end{split}
\]
and
\[
\begin{split}
F^*\omega(e_3,e_6)	=	& -a_1  s_1  (\alpha_{14}  s_3 +\alpha_{16} ) + a_1  s_4 (\alpha_{25}  s_2+\alpha_{35}) + a_3  s_4 (\alpha_{15}  s_1+\alpha_{35} )- a_3 s_2 (\alpha_{24}  s_3+\alpha_{26} ) \\
					& -(a_5  s_3 +a_6)(\alpha_{14}  s_1+\alpha_{24}  s_2+\alpha_{34})+ (a_5  s_4- a_8) (\alpha_{15}  s_1+\alpha_{25}  s_2+\alpha_{35}) \\
					&-a_2  (\alpha_{14}  s_3+\alpha_{15}  s_4+\alpha_{16} )-a_4  (\alpha_{24}  s_3+\alpha_{25}  s_4+\alpha_{26} ),
					\end{split}
\]
we see that the latter is zero if the coefficients of $f^{136}$, $f^{236}$, $f^{346}$, $f^{356}$ in the expression of $F^*\psip$ vanish, 
namely if the following linear system in $s_1,s_2,s_3,s_4$ is compatible
\[
\left\{\begin{array} {l}
\alpha_{14}  s_3+ \alpha_{15}  s_4=-\alpha_{16} , \\
\alpha_{24}  s_3+\alpha_{25}  s_4=-\alpha_{26} , \\
\alpha_{14}  s_1+\alpha_{24}  s_2=-\alpha_{34} , \\
\alpha_{15}  s_1+\alpha_{25}  s_2=-\alpha_{35}. 
\end{array}
\right.
\]
Under the assumption $\lambda(\psip)=16 \, \alpha_{14}\, \alpha_{15}\, \alpha_{24}\, \alpha_{25} <0$, the system has a unique solution $(\bar{s}_1,\bar{s}_2, \bar{s}_3,\bar{s}_4)$. 
Let $\bar{F}\in\mathrm{Aut}(\frs)$ be the automorphism corresponding to the choice $s_i=\bar{s}_i$, for $i=1,2,3,4$. Then, 
\[
\bar{F}^*\psip = \alpha_{25}  e^{235} -\alpha_{25}  e^{256} +\alpha_{24}  e^{234} + \alpha_{24}  e^{246} -\alpha_{15}  e^{135}- \alpha_{15}  e^{156}-\alpha_{14}  e^{134}+ \alpha_{14}  e^{146},
\]
and $\bar{F}^*\omega(e_3,e_6)=0$. A computation then shows that  $J_{\bar{F}^*\psip} e_3 \in \la e_3, e_6 \ra$, and the claim follows.  
\end{proof}

\begin{proposition}
The Lie algebra $\frg_{5,33}^{-1,-1}\oplus\R$ cannot occur as an ideal of a strongly unimodular solvable Lie algebra admitting exact $\G_2$-structures. 
\end{proposition}

\begin{proof}

Let $\frs = \frg_{5,33}^{-1,-1}\oplus\R$ and let $\mathcal{B}^*=(e^1,\ldots,e^6)$ be the basis of $\frs^*$ for which the structure equations are 
\[
(-e^{14}, -e^{25},e^{34}+e^{35},0, 0,0). 
\]
Let $\mathcal{B}=(e_1,\ldots,e_6)$ be the basis of $\frs$ with dual basis $\mathcal{B}^*$. 
Then, the nilradical of $\frs$ is the abelian ideal $\frn =\langle e_1,e_2,e_3,e_6\rangle$, and the generic derivation $D\in\Der(\frs)$ for which $\frg = \frs\rtimes_D\R$ is strongly unimodular 
must have the following matrix representation with respect to $\mathcal{B}$
\[
D = \left(
\begin{array}{cccccc}
a_{1} & 0 & 0 & a_{2}& 0 & 0 
\\
 0 & a_{3} & 0 & 0 & a_{4} & 0 
\\
 0 & 0 & a_{5} & a_{6} & a_{6}& 0 
\\
 0 & 0 & 0 & 0 & 0 & 0 
\\
 0 & 0 & 0 & 0 &0&0
\\
 0 & 0 & 0 & a_7 & a_{8} &  -a_{5}-a_{3}-a_{1}  
\end{array}\right),
\]
where $a_i\in\R$.

We consider a generic 2-form $\alpha = \sum_{1\leq i<j\leq6} \alpha_{ij}e^{ij}$ and a generic 1-form $\beta = \sum_{i=1}^6\beta_k e^k$ on $\frs$, and we let 
$\omega = \hat{d}\beta + D^*\alpha$ and $\psi = \hat{d}\alpha$.
Then 
\[
\begin{split}
\omega =& -\alpha_{12} \left(a_{1}+a_{3}\right) e^{12} -\alpha_{13} \left(a_{1}+a_{5}\right) e^{13}+\left(-\alpha_{14} a_{1}-\alpha_{13} a_{6}-\alpha_{16} a_{7}-\beta_{1}\right) e^{14} \\
		& 	+\left(-\alpha_{15} a_{1}-\alpha_{12} a_{4}-\alpha_{13} a_{6}-\alpha_{16} a_{8}\right) e^{15}+\alpha_{16} \left(a_{3}+a_{5}\right) e^{16} 
			-\alpha_{23} \left(a_{3}+a_{5}\right) e^{23} \\ 
		&	+\left(\alpha_{12} a_{2}-\alpha_{24} a_{3}-\alpha_{23} a_{6}-\alpha_{26} a_{7}\right) e^{24} 
			+\left(-\alpha_{25} a_{3}-\alpha_{23} a_{6}-\alpha_{26} a_{8}-\beta_{2}\right)  e^{25} \\
		&	+\alpha_{26} \left(a_{1}+a_{5}\right) e^{26} +\left(\alpha_{13} a_{2}-\alpha_{34} a_{5}-\alpha_{36} a_{7}+\beta_{3}\right) e^{34} +\alpha_{36} \left(a_{1}+a_{3}\right) e^{36} \\
		&	+\left(\alpha_{23} a_{4}-\alpha_{35} a_{5}-\alpha_{36} a_{8}+\beta_{3}\right) e^{35} 
			 +\left(\alpha_{46} a_{1}-\alpha_{16} a_{2}+\alpha_{46} a_{3}+\alpha_{46} a_{5}-\alpha_{36} a_{6}\right) e^{46}\\
		&	+\left(-\alpha_{15} a_{2}+\alpha_{24} a_{4}+\alpha_{34} a_{6}-\alpha_{35} a_{6}+\alpha_{56} a_{7}-\alpha_{46} a_{8}\right) e^{45} \\
		&	+\left(\alpha_{56} a_{1}+\alpha_{56} a_{3}-\alpha_{26} a_{4}+\alpha_{56} a_{5}-\alpha_{36} a_{6}\right) e^{56}, 
\end{split}
\]
and
\[
\begin{split}
\psip 	&=   \alpha_{12}  (e^{125} + e^{124})  - \alpha_{13} e^{135}  - \alpha_{15} e^{145} - \alpha_{16} e^{146}  -\alpha_{23} e^{234} + \alpha_{24} e^{245}  - \alpha_{26} e^{256}   \\
 		&\quad + (\alpha_{35} -\alpha_{34}) e^{345} + \alpha_{36} (e^{346} + e^{356} ). 
\end{split}
\]
We will prove that there are no values of the parameters $a_r,\alpha_{ij},\beta_k$ for which the pair $(\omega, \psip)$ defines an $\SU(3)$-structure. 

Let us assume that $\omega^3\neq0$ and 
\[
\lambda(\psip) = 4 \left( \alpha_{26} \alpha_{36} \alpha_{12} \alpha_{13} - \alpha_{23} \alpha_{36} \alpha_{12} \alpha_{16} + \alpha_{13} \alpha_{16} \alpha_{23} \alpha_{26}\right) <0. 
\]
Up to changing the sign of $\omega$, we can assume that the corresponding orientation is $e^{123456}$. 
We can then determine the almost complex structure $J$ induced by the pair $(\omega,\psip)$ and consider the bilinear form $g = \omega(\cdot,J\cdot)$. 
Let $g_{ij}\coloneqq g(e_i,e_j)$ be the components of the matrix associated with $g$ with respect to the basis $\mathcal{B}$. 
We must have 
\begin{enumerate}[i)]
\item $\alpha_{12} \, \alpha_{13}  \, \alpha_{16} \, \alpha_{23}  \, \alpha_{26} \, \alpha_{36} \neq 0$,
\end{enumerate}
as otherwise $g_{ii}=0$, for at least one $i \in \{1,2,3,6\}$.

\smallskip

We now focus on the compatibility condition $\omega\W\psip=0$, which is equivalent to a system of five polynomial equations in the variables  $a_r,\alpha_{ij},\beta_k$. 
We compute
\[
\begin{split}
(\o \wedge \psip)(e_1,e_2,e_3,e_4,e_6) &= 2(\alpha_{16} \alpha_{23}(a_3+a_5) -\alpha_{12}\alpha_{36}(a_1+a_3) ) \eqqcolon z_1 ,\\
(\o \wedge \psip)(e_1,e_2,e_3,e_5,e_6) &= -2(\alpha_{13}  \alpha_{26} (a_1+a_5)+ \alpha_{12} \alpha_{36} (a_1+a_3)) \eqqcolon z_2,
\end{split}
\]
and we use the equalities
\[
\begin{split}
&\dfrac{g_{12}}{\alpha_{12}} = -\dfrac{g_{21}}{\alpha_{12}} =  \dfrac{g_{63}}{\alpha_{36}}=  - \dfrac{g_{36}}{\alpha_{36}}= \left(z_1-z_2\right)\sqrt{-\lambda(\psip)}, \\ 
&\dfrac{g_{13}}{\alpha_{13}} = -\dfrac{g_{31}}{\alpha_{13}} = \dfrac{g_{62}}{\alpha_{26}} = -\dfrac{g_{26}}{\alpha_{26}} = -z_1\,\sqrt{-\lambda(\psip)},  \\ 
&\dfrac{g_{16}}{\alpha_{16}} = -\dfrac{g_{61}}{\alpha_{16}} = \dfrac{g_{32}}{\alpha_{23}} =  -\dfrac{g_{23}}{\alpha_{23}} = -z_2.\,\sqrt{-\lambda(\psip)},
\end{split}
\]
to conclude that every solution of the system must give
\begin{enumerate}[i)]
\item[ii)] $g(e_i,e_j)=0$, for $i,j\in\{1,2,3,6\}$ with $i\neq j$. 
\end{enumerate}

The three equations 
\[
(\omega \wedge \psip)(e_2,e_3,e_4,e_5,e_6)=0,\quad (\omega \wedge \psip)(e_1,e_3,e_4,e_5,e_6)=0,\quad (\omega \wedge \psip)(e_1,e_2,e_4,e_5,e_6)=0,
\]
determine a linear system in the variables $\beta_1,\beta_2,\beta_3$, which has a unique solution under the constraint i). 
The remaining equations 
\begin{equation} \label{sysAandB}
\left \{ \begin{array} {l} 
(\omega \wedge \psip)(e_1,e_2,e_3,e_4,e_5)=0,\\
(\omega \wedge \psip)(e_1,e_2,e_3,e_4,e_6)=0,\\
(\omega \wedge \psip)(e_1,e_2,e_3,e_5,e_6)=0,
\end{array}
\right.
\end{equation}
do not contain the variables $\beta_k$'s, and we can solve the system \eqref{sysAandB} in the following cases:
\begin{enumerate}[A)]
\item\label{caseA} $\alpha_{12}  \alpha_{13} \alpha_{26} \alpha_{36}+   \alpha_{12} \alpha_{16}  \alpha_{23}\alpha_{36}  + \alpha_{13}  \alpha_{16} \alpha_{23} \alpha_{26}  \neq 0$,  
namely  \eqref{sysAandB} is  a compatible linear system in the variables $a_1,a_2,a_5$;
\item\label{caseB} $ \alpha_{12}  \alpha_{13}  \alpha_{26} \alpha_{36}  -  \alpha_{12}  \alpha_{16}  \alpha_{23} \alpha_{36} - \alpha_{13}  \alpha_{16}  \alpha_{23} \alpha_{26}  \neq 0$,  
namely \eqref{sysAandB}  is a compatible linear system in the variables $a_1, a_2, a_3$.
\end{enumerate}
Indeed, if the two polynomials above are both zero, then $\alpha_{12}  \, \alpha_{13}  \, \alpha_{26}  \, \alpha_{36} =0$, which contradicts the condition i). 
In both cases \ref{caseA}) and \ref{caseB}), we can use the constraint to solve the system \eqref{sysAandB}. 
Then, we have $\omega\W\psip=0$, and the bilinear form $g= \omega(\cdot,J\cdot)$ is symmetric. 
We now show that $g$ is never positive definite.

To simplify the computations, we can proceed as follows. 
Let us consider the derivation $S\in\Der(\frs)$ whose matrix with respect to the basis $\mathcal{B}$ is 
\[
\left( \begin{array}{cccccc} 
0&0& 0 & s_1 & 0 & 0 \\[2pt]
 0 & 0 &  0 & 0 &  s_2 & 0\\[2pt]
0 &  0 &  0 & s_3 &s_3& 0\\[2pt]
0 &  0 &  0 &  0 & 0 &  0\\[2pt]
 0 & 0 &  0 &  0 &  0 &  0\\[2pt]
 0 & 0 &  0 &  0 &  0 &  0 
 \end{array}
 \right ),
\]
where
\[
\begin{array}{c} 
s_1=\dfrac{\alpha_{12}  \alpha_{34}  - \alpha_{12}  \alpha_{35} + \alpha_{13} \alpha_{24}  -\alpha_{15} \alpha_{23}}{2  \alpha_{12} \alpha_{13}},\\[8pt] 
s_2= -\dfrac{\alpha_{12} \alpha_{34}   - \alpha_{12}  \alpha_{35}- \alpha_{13}  \alpha_{24} + \alpha_{15} \alpha_{23}}{2 \alpha_{12} \alpha_{13}}, \\[8pt]
s_3=\dfrac{\alpha_{12}  \alpha_{34}   - \alpha_{12}  \alpha_{35}  - \alpha_{13} \alpha_{24}  - \alpha_{15} \alpha_{23}}{2  \alpha_{12} \alpha_{13} }.
\end{array}
\]
Then, the automorphism $F=\exp(S)\in\Aut(\frs)$ is such that
\[
F^* \psip = \alpha_{12} (e^{124}+ e^{125}) - \alpha_{13}  e^{135}  - \alpha_{16}  e^{146}   - \alpha_{23}  e^{234} - \alpha_{26}  e^{256}  + \alpha_{36}\left(e^{346} + e^{356}\right).
\]
Notice that $\lambda\left(F^*\psi\right) = \lambda\left(\psi\right)$.
Moreover, this choice of $F$ guarantees that the condition ii) is satisfied also by the bilinear form $F^*g = F^*\omega(\cdot,J_{F^*\psip}\cdot)$, 
where $J_{F^*\psip} = F^{-1}\circ J\circ F$.  In addition, the subspaces $V_1=\la e_1,e_2,e_3,e_6 \ra$ and $V_2=\la e_4,e_5\ra$ are $J_{F^*\psip}$-invariant.

Let  $Q$  be the matrix associated to $F^* g$ with respect to the basis $\mathcal{B}$. 
We will show that there are no values of the parameters $a_r,\alpha_{ij},\beta_k$ for which $Q$ is symmetric and positive definite. 
If that was the case, then it would be possible to construct a $F^*g$-orthonormal basis $(v_1,\ldots,v_6)$ starting from $\mathcal{B}$ in such a way that 
\[
v_i = h_{ii}\,e_i,~i=1,2,3,6,\quad  v_4 = \sum_{i=1}^4 h_{i4}e_i + h_{64}e_6,\quad v_5 = \sum_{i=1}^6 h_{i5}e_i,
\]
where $h_{kk}>0$, for $1\leq k \leq6$. 
Consequently, there would exist an invertible $6\times6$ matrix $P=(p_{ij})$, given by the inverse of $H=(h_{ij})$, such that 
\[
Q=P^tP,
\]
and whose entries satisfy the following conditions
\begin{equation}\label{Pentcond}
p_{kk}>0,\quad p_{54}=0,\quad p_{ki} = 0, \mbox{ for } 1\leq k \leq 6 \mbox{ and } i\neq k \mbox{ in } \{1,2,3,6\}.
\end{equation}
Moreover, the following quantities should all be positive  
\begin{equation}\label{Qposdef}
\frac{F^*g(e_1,e_1)}{F^*g(e_6,e_6)} = -\frac{\alpha_{12}\alpha_{13}}{\alpha_{26}\alpha_{36}},\quad 
 \frac{F^*g(e_2,e_2)}{F^*g(e_6,e_6)} = \frac{\alpha_{12}\alpha_{23}}{\alpha_{16}\alpha_{36}},\quad 
 \frac{F^*g(e_3,e_3)}{F^*g(e_6,e_6)} = -\frac{\alpha_{13}\alpha_{23}}{\alpha_{16}\alpha_{26}}.
%
\end{equation}
Let $\mathscr{E}\subset\R^6$ be the (non-empty) open subset where all of the previous conditions hold. Notice that condition i) is satisfied by every  $6$-tuple 
$(\alpha_{12}, \alpha_{13}, \alpha_{16}, \alpha_{23}, \alpha_{26}, \alpha_{36} )\in\mathscr{E}$. 

\smallskip

Since $(J_{F^*\psip})^tQ = - QJ_{F^*\psip}$, we determine all $6\times6$ invertible matrices $P=(p_{ij})$ whose entries satisfy the conditions \eqref{Pentcond} and for which 
\[
Z\coloneqq (J_{F^*\psip})^t (P^tP) + (P^tP) J_{F^*\psip}
\] 
is the zero matrix. 
On $\mathscr{E}$, this boils down to solving a system of 17 equations in the unknowns $p_{ij}$ under the constraints \eqref{Pentcond}. 
The sub-system  $\{Z_{ij}=0\st i,j=1,2,3,6,~i<j\}$ has the following solution 
\[
p_{11} = p_{66}\sqrt{-\frac{\alpha_{12}\alpha_{13}}{\alpha_{26}\alpha_{36}}},\quad p_{22} = p_{66}\sqrt{\frac{\alpha_{12}\alpha_{23}}{\alpha_{16}\alpha_{36}}},\quad 
p_{33} = p_{66}\sqrt{ -\frac{\alpha_{13}\alpha_{23}}{\alpha_{16}\alpha_{26}}}, 
\]
for any choice of $p_{66}>0$. 
The positivity of the quantities in \eqref{Qposdef} together with conditions \eqref{Pentcond} ensure that the sub-system $\{Z_{i4}=0\st i=1,2,3,6\}$ can be solved with respect to the unknowns $p_{i4}$, for $i=1,2,3,6$, 
and one has that the solution 
also solves the sub-system $\{Z_{i5}=0\st i=1,2,3,6\}$. We are then left with the equations $Z_{44}=0,~Z_{45}=0,~Z_{55}=0$. The equation $Z_{44}=0$ has the following solution 
\[
p_{44} = \left(1+\frac{\alpha_{16}\alpha_{23}}{\alpha_{12}\alpha_{36} }\right) p_{45}, 
\]
and by the positivity of the quantities in \eqref{Qposdef} we must have $p_{45}>0$. 
Finally, the equations $Z_{45}=0=Z_{55}$ hold if and only if
\[
\left(\alpha_{12}\alpha_{36}\right)^2 p_{55}^2 + \frac{\lambda\left(F^*\psi\right)}{4} p_{45}^2 = 0, 
\] 
and we thus obtain
\[
p_{55} = \frac{\sqrt{-\lambda\left(F^*\psi\right)}}{2 |\alpha_{12}\alpha_{36}|} p_{45},
\]

Summing up, when $(\alpha_{12}, \alpha_{13}, \alpha_{16}, \alpha_{23}, \alpha_{26}, \alpha_{36} )\in\mathscr{E}$, 
then all $6\times6$ matrices $P$ satisfying the conditions \eqref{Pentcond} and $ (J_{F^*\psip})^t (P^tP) + (P^tP) J_{F^*\psip}=0$ 
constitute a family  $\mathscr{P}$ of matrices  depending on two positive real parameters $p_{66}$ and $p_{45}$, and on real parameters $p_{15},p_{25},p_{35},p_{65}$.

\smallskip

Now, if $Q$ was symmetric and positive definite, then there would exist a matrix $P\in\mathscr{P}$  such that $Q = P^tP$. In particular, the following identity should hold 
\begin{equation}\label{QJP}
QJ_{F^*\psip} = P^tPJ_{F^*\psip}.
\end{equation}
Assume that $(\alpha_{12}, \alpha_{13}, \alpha_{16}, \alpha_{23}, \alpha_{26}, \alpha_{36} )\in\mathscr{E}$ is given. 
In the cases \ref{caseA}) and \ref{caseB}), which ensure that $Q$ is symmetric, we consider the system of equations corresponding to the matrix identity \eqref{QJP}, 
where $P$ is any  matrix in $\mathscr{P}$. 
This consists in 15 equations in the unknowns $\alpha_{15}$, $\alpha_{24}$, $\alpha_{34}$, $\alpha_{35}$, $\alpha_{46}$, $\alpha_{56}$, 
$a_4$, $a_6$, $a_7$, $a_8$, and $a_3$, in case \ref{caseA}), or $a_5$, in case \ref{caseB}).  
With the aid of a computer algebra system, it is possible to show that there are no values of the unknowns for which the system can be solved. This gives a contradiction.  

For the reader's convenience, we now describe the relevant steps leading to the conclusion.   
Let $M\coloneqq QJ_{F^*\psip}  - P^tPJ_{F^*\psip}$. In case \ref{caseA}), the entries $M_{ij}$, for $i,j\in \{1,2,3,6\}$ with $i<j$, are all proportional to the same polynomial and one has $M_{ij}=0$ if and only if 
\[
a_3 = \frac{p_{66}^2}{\alpha_{16}\alpha_{26}\alpha_{36}}
	\frac{\alpha_{12}  \alpha_{13} \alpha_{26} \alpha_{36}+   \alpha_{12} \alpha_{16}  \alpha_{23}\alpha_{36}  + \alpha_{13}  \alpha_{16} \alpha_{23} \alpha_{26} }{\lambda\left(F^*\psi\right)}.
\]
We now consider the equations $M_{i4}=0$ and $M_{i5}=0$, $i\in\{1,2,3,6\}$, which can be seen as a linear system of eight equations in the unknowns $a_4$, $a_6$, $a_7$ and $a_8$. 
This system admits a unique solution on $\mathscr{E}$.  
We are left with the equation $M_{45}=0$.  
On $\mathscr{E}$, $M_{45}$ can be seen as a second degree polynomial in the unknowns $\alpha_{15}$, $\alpha_{24}$, $\alpha_{34}$, $\alpha_{35}$, $\alpha_{46}$, $\alpha_{56}$. 
We claim that $M_{45}$ is always non-zero. 
Thinking of $M_{45}$ as a polynomial in $\alpha_{15}$, we first compute its discriminant $\Delta_1$, which is a second degree polynomial in the remaining unknowns. 
To show the claim, we think of $\Delta_1$ as a second degree polynomial in $\alpha_{24}$ and we prove that it is always negative. 
First, we observe that the leading coefficient of $\Delta_1$ is negative on $\mathscr{E}$. Indeed, its sign is determined by  
\[
-\alpha_{16}\alpha_{23}\left(\alpha_{12}\alpha_{36} - \alpha_{13}\alpha_{26} + \alpha_{16}\alpha_{23}\right), 
\]
and the quantity inside the brackets has the same sign of $\alpha_{16}\alpha_{23}$ on $\mathscr{E}$. 
The discriminant $\Delta_2$ of $\Delta_1$ can be seen as a quadratic form in $p_{15}$, $p_{25}$, $p_{35}$, $p_{45}$, $p_{65}$, which is negative definite on $\mathscr{E}$ since $p_{45}>0$.  
Therefore, $\Delta_1<0$. 
An analogous discussion shows the thesis also in case \ref{caseB}).

\end{proof}

\subsection{{The indecomposable case}}
We now consider six-dimensional {\em indecomposable} unimodular solvable non-nilpotent Lie algebras. 
Their structure equations with respect to a suitable basis $(X_1,\ldots,X_6)$ can be found in the tables A.9--A.19 of \cite{Boc}, where the Lie algebras are gathered together according to their nilradical.
Notice that there are a few misprints in \cite{Boc} that must be corrected as follows:      
\begin{enumerate}[$\bullet$] 
\item Lie algebra $\frg_{6,55}^{-4}$ of Table A.12: $[X_3,X_6]=5 X_3$, $[X_5,X_6] = -4 X_5$;
\item Lie algebra $\frg_{6,83}^{0,l}$ of Table A.15: $[X_2,X_6] = l X_2+X_3$;
\item Lie algebra $\frg_{6,135}^{0,-4}$ of Table A.19: $[X_3,X_5]=X_2$. 
\end{enumerate}
In the following, we will keep on denoting the basis of a Lie algebra by $(e_1,\ldots,e_6)$ and the corresponding dual basis by $(e^1,\ldots,e^6)$.

\smallskip

The next general result rules out the Lie algebras listed in tables A.$9$, A.$10$, A.$15$ of \cite{Boc}. 

\begin{proposition}\label{IndecLambda}
Let $\frs$ be a six-dimensional 
unimodular solvable non-nilpotent Lie algebra, and denote by $\frn$ its nilradical. 
Then, every exact $3$-form $\hat{d}\alpha\in\Lambda^3\frs^*$ is not stable, whenever $\frn$ is isomorphic to one of the Lie algebras $\R^5$, $\frg_{3,1}\oplus\R^2$, $\frg_{5,4}$, 
where $\frg_{3,1}$ and $\frg_{5,4}$ denote the three-dimensional and the five-dimensional Heisenberg Lie algebra, respectively. 
\end{proposition}

\begin{proof}   
We have $\frs\cong\frn\rtimes_S\R$, where $S\in\Der(\frn)$ is a derivation of the five-dimensional nilpotent ideal $\frn$. 
We choose a basis $(e_1,\ldots,e_6)$ of $\frs$ so that $\frn = \langle e_1,\ldots,e_5\rangle$ and $\R = \langle e_6\rangle$. Then, we can write every $2$-form $\alpha\in\Lambda^2\frs^*$ with respect to the dual basis 
$(e^1,\ldots,e^6)$ as follows 
\[
\alpha = \alpha_\frn + \alpha'\W e^6,
\]
where $\alpha_\frn\in\Lambda^2\frn^*$ and $\alpha'\in\frn^*$. We will prove that $\hat{d}\alpha$ is never stable by showing the existence of a non-zero vector $x\in\frn$ for which  
the 2-form $\iota_x\hat{d}\alpha$ has rank at most two.

When $\frs$ is almost abelian, namely $\frn\cong\R^5$,  the Chevalley-Eilenberg differential of $\alpha$ is given by 
\[
\hat{d}\alpha = S^*\alpha_\frn\W e^6. 
\]
Thus, for every non-zero vector $x\in\frn$ we have $\iota_x\hat{d}\alpha = (\iota_xS^*\alpha_\frn)\W e^6$. Consequently, $\iota_x\hat{d}\alpha\W \iota_x\hat{d}\alpha=0$ and the claim follows.  

%

Let $d_\frn$ denote the Chevalley-Eilenberg differential of $\frn$. When $\frn\cong\frg_{3,1}\oplus\R^2$ or $\frn\cong\frg_{5,4}$, 
we can choose the basis $(e^1,\ldots, e^5)$ of $\frn^*$ in such a way that $d_\frn e^k=0$, for $k=2,3,4,5$, and $d_\frn e^1 = e^{23}$ when $\frn\cong\frg_{3,1}\oplus\R^2$ while 
$d_\frn e^1 = e^{24}+e^{35}$ when $\frn\cong\frg_{5,4}$. In both cases, we then obtain 
\[
\iota_{e_1}\hat{d}\alpha = \left(\iota_{e_1}S^*\alpha_\frn\right) \W e^6, 
\]
and from this we see that $\iota_{e_1}\hat{d}\alpha$ has rank at most two.  
\end{proof}

For every Lie algebra $\frs$ not isomorphic to one of those considered in the previous proposition, 
we first have to compute the expression of the generic derivation $D\in\Der(\frs)$, consider the extension $\frg = \frs\rtimes_D\R$, and determine for which derivations $D$ it is strongly unimodular. 
Then, we have to show that there are no pairs $(\omega,\psip)$ of the form \eqref{SU3ideal} defining an SU(3)-structure on $\frs$. 
We shall deal with this problem in the next propositions. 

\begin{proposition}\label{PropTabA11}
The  indecomposable unimodular solvable non-nilpotent Lie algebras listed in \cite[Table $\mathrm{A}.11$]{Boc} $(\mbox{nilradical } \frg_{4,1}\oplus\R)$ 
cannot occur as an ideal of a strongly unimodular solvable Lie algebra admitting exact $\G_2$-structures. 
\end{proposition}
\begin{proof}
Let $\frs$ denote one of the Lie algebras listed in Table $\mathrm{A}.11$ of \cite{Boc}, and let $(e_1,\ldots,e_6)$ be the basis of $\frs$ for which the structure equations are those given in that table. 
The nilradical of $\frs$ is $\frn=\langle e_1,e_2,e_3,e_4,e_5\rangle\cong \frg_{4,1}\oplus\R$, and its descending central series is 
\[
\frn^{\sst0}=\frn,\quad \frn^{\sst1} = \langle e_1,e_2\rangle,\quad \frn^{\sst2}= \langle e_2\rangle,\quad \frn^{\sst i}=\{0\}, i\geq3.
\]
Let $D$ be a generic derivation of $\frs$, and consider the extension $\frg = \frs\rtimes_D\R$. 
Then, a computation shows that $\frg$ is strongly unimodular only when $e_2\in\ker D$ and the image of the restriction of $D$ to the subspace $\langle e_1,e_3\rangle\subset \frs$  is $\langle e_2\rangle$. 
Now, we determine the expression of the almost complex structure $J$ induced by a generic negative stable exact 3-form $\hat{d}\alpha$ and the volume form $e^{123456}$, and we observe that 
$Je_2\in\langle e_1,e_2,e_3\rangle$. Since  $[e_1,e_2]=0=[e_2,e_3]$, we see that 
\[
\omega(e_2,Je_2) = (\hat{d}\beta+D^*\alpha)(e_2,Je_2) = -\beta([e_2,Je_2]) - \alpha(De_2,Je_2) - \alpha(e_2,DJe_2) =0.
\]
Since the previous discussion holds for every $\alpha\in\Lambda^2\frs^*$ and $\beta\in\frs^*$ such that $\omega^3\neq0$ and $\hat{d}\alpha$ is stable, the thesis follows. 
\end{proof}

\begin{proposition}
The  six-dimensional indecomposable unimodular solvable non-nilpotent Lie algebras listed in 
Table $\mathrm{A}.12$ $(\mbox{nilradical } \frg_{5,1})$, $\mathrm{A}.13$ $(\mbox{nilradical } \frg_{5,2})$, $\mathrm{A}.14$ $(\mbox{nilradical } \frg_{5,3})$, 
$\mathrm{A}.16$ $(\mbox{nilradical } \frg_{5,5})$, $\mathrm{A}.18$ and $\mathrm{A}.19$  $(\mbox{nilradical } \frg_{3,1}\oplus\R)$ 
of \cite{Boc} cannot occur as an ideal of a strongly unimodular solvable Lie algebra admitting exact $\G_2$-structures. 
\end{proposition}
\begin{proof}
Among the Lie algebras mentioned in the statement, $\frg_{6,54}^{2+2l,l}$, $\frg_{6,70}^{4p,p}$ and $\frg_{6,65}^{4l,l}$ (all with nilradical $\frg_{5,1}$) are the only ones depending on a real parameter.  
When $\frs$ is one of these Lie algebras, then $\frg = \frs\rtimes_D\R$ is strongly unimodular only for a certain value of the parameter. 
As we will see, this fact will be relevant to rule out two of them, namely $\frg_{6,54}^{2+2l,l}$ and $\frg_{6,70}^{4p,p}$. 
Thus, we begin assuming that $\frs$ is any of the Lie algebras considered in the statement, with the exception of $\frg_{6,54}^{2+2l,l}$ and $\frg_{6,70}^{4p,p}$. 

For each Lie algebra $\frs$, we consider the basis $(e_1,\ldots,e_6)$ for which the structure equations are those given in \cite{Boc}, 
and the generic derivation $D\in\Der(\frs)$ such that $\frg= \frs\rtimes_D\R$ is strongly unimodular. 
We fix the volume form $e^{123456}$, and we compute the almost complex structure $J$ induced by the generic negative stable 3-form $\psip=\hat{d}\alpha$. 
Then, arguing as in the proof of Proposition \ref{PropTabA11}, we consider the generic non degenerate 2-form $\omega = \hat{d}\beta + D^*\alpha$, 
and we observe that the properties of $D$ and $J$ allow us to single out  (at least) one basis vector $e_i$ in the nilradical of $\frs$ such that $\omega(e_i,Je_i)=0$. 
For the sake of clarity, we give the details in the case when $\frs$ is isomorphic to the Lie algebra $\frg_{6,55}^{-4}$. 
The discussion in the remaining cases is similar, but it depends both on the specific expressions of $D$ and $J$ one obtains and on the structure equations of the nilradical of $\frs$. 

The nilradical of  $\frg_{6,55}^{-4}$ is $\frg_{5,1}$, and we can choose a basis $(e_1,\ldots,e_5)$ of it  in such a way that the only non-zero brackets are 
$[e_3,e_5] =e_1 $ and $[e_4,e_5]=e_2$. Now, a generic derivation $D$ of $\frs=\frg_{6,55}^{-4}$ for which $\frg = \frs\rtimes_D\R$ is strongly unimodular satisfies 
$De_3\in\langle e_1\rangle$ and $\langle e_1,e_2\rangle\subset \ker D$. 
Moreover, the almost complex structure $J$ induced by a generic negative stable exact 3-form $\hat{d}\alpha$ and the volume form $e^{123456}$ satisfies $Je_1\in\langle e_1,e_2,e_3\rangle$. 
Therefore, we have
\[
\omega(e_1,Je_1) = -\beta([e_1,Je_1]) - \alpha(De_1,Je_1) - \alpha(e_1,DJe_1) =0,
\] 
and the claim follows. 

We are then left with the Lie algebras $\frg_{6,70}^{4p,p}$ and $\frg_{6,54}^{2+2l,l}$, where a different type of approach is needed in order to rule them out. 

If $\frs = \frg_{6,70}^{4p,p}$, we consider the basis $(e^1,\ldots,e^6)$ of $\frs^*$ for which the structure equations are the following
\[
\left( e^{26} -p\,e^{16} - e^{35}, -e^{16}-p\,e^{26}-e^{45}, 3p\,e^{36}+e^{46},3p\,e^{46}-e^{36},-4p\,e^{56},0 \right),\quad p\in\R. 
\]
The nilradical of $\frs$ is $\frn = \langle e_1,e_2,e_3,e_4,e_5\rangle \cong \frg_{5,1}$, and the only non-trivial terms in its descending central series are 
\[
\frn^{\sst0} = \frn,\quad \frn^{\sst1} = \langle e_1,e_2\rangle. 
\]
Moreover, the generic derivation of $\frs$ has the following matrix representation
\[
D = \left(\begin{array}{cccccc}
a_{1} & a_{2} & a_{3} & 0 & a_{4} & a_{5} 
\\
 -a_{2} & a_{1} & 0 & a_{3} & a_{6} & a_{7} 
\\
 0 & 0 & a_{8} & a_{2} & 0 & -3 p a_{4}-a_{6} 
\\
 0 & 0 & -a_{2} & a_{8} & 0 & -3 p a_{6}+a_{4} 
\\
 0 & 0 & 0 & 0 & -a_{8}+a_{1} & -4p a_{3}   
\\
 0 & 0 & 0 & 0 & 0 & 0 
\end{array}\right).
\]
In order to obtain a strongly unimodular extension $\frg = \frs\rtimes_D\R$, we must have $a_1 = 0 = a_8$ and $0 = \tr(\ad_{e_6}|_{\frn^{\sst1}}) = 2p$  
(notice that the nilradical of $\frg$ coincides with $\frn$ when $a_1 = 0 = a_8$). 
We can then conclude observing that the quartic polynomial associated with the generic exact 3-form $\hat{d}\alpha$ on $\frs$ is 
\[
\begin{split}
\lambda(\hat{d}\alpha)	&= 4\, \alpha_{12}^{2} \left(\left(\alpha_{13}-\alpha_{24}\right)^{2} + \left(\alpha_{14}+\alpha_{23}\right)^{2}\right) \\
					&\quad -16\,\alpha_{12}^2\, p^2\left( 3\alpha_{12}\alpha_{34} + \alpha_{13}\alpha_{24} - \alpha_{14}^2 + \alpha_{14}\alpha_{23} - \alpha_{23}^2  \right),
\end{split}
\]
and thus $\lambda(\hat{d}\alpha)\geq0$ when $p=0$.

\smallskip

Finally, we consider $\frs = \frg_{6,54}^{2+2l,l}$, whose structure equations with respect to a basis $(e^1,\ldots,e^6)$ of $\frs^*$ are the following
\[
\left(-e^{16}-e^{35}, -l\,e^{26}-e^{45}, \left(1+2l\right)e^{36},\left(2+l\right)e^{46},-\left(2+2l\right) e^{56},0 \right), \quad l\in\R.
\]
As in the previous case, the nilradical of $\frs$ is $\frn = \langle e_1,e_2,e_3,e_4,e_5\rangle  \cong \frg_{5,1}$. 
Let $D\in\Der(\frs)$ be a generic derivation and consider the extension $\frg =  \frs\rtimes_D\R$. Requiring $\frg$ to be strongly unimodular gives
\[
D =\left(
\begin{array}{cccccc}
a_{1} & 0 & a_{2} & 0 & a_{3} & a_{4} 
\\
 0 & -a_{1} & 0 & a_{2} & a_{6} & a_{7} 
\\
 0 & 0 & a_{1} & 0 & 0 & a_{3} \left(-1-2 l \right) 
\\
 0 & 0 & 0 & -a_{1} & 0 & a_{6} \left(-2-l \right) 
\\
 0 & 0 & 0 & 0 & 0 & a_{2} \left(-2-2 l \right) 
\\
 0 & 0 & 0 & 0 & 0 & 0 
\end{array}\right),
\]
and $0 = \tr(\ad_{e_6}|_{\frn^{\sst1}}) = l+1$. 
We now compute the generic 2-form $\omega = \hat{d}\beta-D^*\alpha$ and the generic 3-form $\psip=\hat{d}\alpha$, and we consider all possible values of the parameters $a_r$, $\alpha_{ij}$, $\beta_k$ 
for which $\omega^3\neq0$ and 
\[
\lambda(\psip) = 16 \alpha_{12}^{2} \alpha_{24} \alpha_{13} <0. 
\]
We then determine the almost complex structure $J$ induced by $(\omega,\psip)$, and we conclude observing that  $\omega(e_i,Je_i)=0$, for $i=1,2,3,4$. 
Notice that, when $l+1\neq0$, the Lie algebra $\frg$ is not strongly unimodular and $\omega(e_i,Je_i)$ is proportional to $l+1$, for $i=1,2,3,4$. 
\end{proof}

We still have to examine the indecomposable Lie algebras listed in \cite[Table A.17]{Boc}, namely those with abelian nilradical $\R^4$.  
In the next result, we rule out all of them but $\frak g_{6,101}^{a,b,-1-a,-1-b}$. This last Lie algebra will be considered in Proposition \ref{6101}.

\begin{proposition}
Let $\frs$ be a Lie algebra that is isomorphic to one of those listed in Table $\mathrm{A}.17$ of \cite{Boc} but $\frak g_{6,101}^{a,b,-1-a,-1-b}$. 
Then, $\frs$  cannot occur as an ideal of a strongly unimodular solvable Lie algebra admitting exact $\G_2$-structures.
\end{proposition}
\begin{proof}
Assume first that $\frs$ is not isomorphic to one of $\frg_{6,114}^{a,-1,-a/2}$, $\frg_{6,115}^{-1,b,c,-c}$, $\frg_{6,118}^{0,b,-1}$. 
Let $(e_1,\ldots,e_6)$ be a basis of $\frs$ for which the structure equations are those given in Table $\mathrm{A}.17$ of \cite{Boc},  
and consider the generic derivation $D\in\Der(\frs)$ such that $\frg = \frs\rtimes_D\R$ is strongly unimodular.
Then, arguing as in the proof of Proposition \ref{PropTabA11}, 
we obtain that for every non-degenerate 2-form $\omega = \hat{d}\beta + D^*\alpha$ and every negative stable 3-form $\psip = \hat{d}\alpha$ there exists (at least) one basis vector $e_i$ 
in the nilradical $\R^4$ of $\frs$ such that 
$\omega(e_i,Je_i)=0$. As before, this depends on the expressions of $D$ and $J$ in each case under exam. 

For the remaining Lie algebras of Table $\mathrm{A}.17$, we obtain different types of contradictions. 
If $\frs$ is one of $\frg_{6,115}^{-1,b,c,-c}$, $\frg_{6,118}^{0,b,-1}$, we can proceed as above and conclude observing that 
\[
\omega(e_3,Je_3)\,\omega(e_4,Je_4) \leq 0.
\]

Finally, if $\frs = \frg_{6,114}^{a,-1,-a/2}$, we consider the basis $(e^1,\ldots,e^6)$ of $\frs^*$ for which the structure equations are 
\[
 \frg_{6,114}^{a,-1,-a/2} = \left(a\,e^{15}-e^{16}, e^{26}, -\frac{a}{2}\,e^{35}-e^{45}, e^{35}+\frac{a}{2}\,e^{45},0,0 \right),\quad a\neq0.
\]
Then, we choose the volume form $e^{123456}$, and we compute the quartic polynomial $\lambda(\hat{d}\alpha)$, for a generic $\alpha=\sum_{1\leq i < j \leq6}\alpha_{ij}e^{ij}\in\Lambda^2\frs^*$,  
obtaining 
\[
\lambda(\hat{d}\alpha) = \left(a \left(\alpha_{13} \alpha_{24}+\alpha_{14} \alpha_{23}\right)-2 \alpha_{13} \alpha_{23}-2 \alpha_{14} \alpha_{24}\right)^{2} 
					+4 \left(-\alpha_{13} \alpha_{24}+\alpha_{14} \alpha_{23}\right)^{2} \geq0.
\]
\end{proof}

\begin{proposition}\label{6101}
The indecomposable Lie algebra $\frg_{6,101}^{a,b,-1-a,-1-b}$ cannot occur as an ideal of a strongly unimodular solvable Lie algebra admitting exact $\G_2$-structures.
\end{proposition}
\begin{proof}
Let  $\frs = \frg_{6,101}^{a,b,-1-a,-1-b}$, and consider the basis $\mathcal{B}^* = (e^1.\ldots,e^6)$ of $\frs^*$ for which the structure equations are 
\[
 (a e^{15} + b e^{16}, - (a + 1) e^{25} - (b +1) e^{26}, e^{36}, e^{45}, 0,0),
\]
with $ab \neq 0$ and $(-a  -1)^2 + (-b -1)^2 \neq 0$. Let $\mathcal{B}=(e_1,\ldots,e_6)$ be the basis of $\frs$ with dual basis $\mathcal{B}^*$. 
We will study the cases $a = -1$ and $a \neq - 1$ separately. 

\smallskip

Assume that $a = -1$. Then, the nilradical of  $\frs$ is the abelian ideal $\frn =\langle e_1,e_2,e_3, e_4\rangle$, 
and the generic derivation $D\in\Der(\frs)$ for which $\frg = \frs \rtimes_D\R$ is strongly unimodular has the following matrix representation with respect to $\mathcal{B}$
\[
D = \left(
\begin{array}{cccccc}
a_{1} & 0 & 0 &0& a_2 & -a_2 \,  b 
\\
 0 & a_{3} & 0 & 0 & 0 &a_4 
\\
 0 & 0 & a_{5} & 0 & 0& a_6
\\
 0 & 0 & 0 & -a_{5}-a_{3}-a_{1}  & a_8 & 0 
\\
 0 & 0 & 0 & 0 &0&0
\\
 0 & 0 & 0 & 0 & 0&  0
\end{array}\right),
\]
where $a_i\in\R$.

Let $\alpha = \sum_{1\leq i<j\leq6} \alpha_{ij}e^{ij}$ and $\beta = \sum_{i=1}^6\beta_k e^k$ be a generic $2$-form  and  a generic $1$-form on  $\frs$, respectively.  
Then, the forms $(\omega,\psip)$ given by \eqref{SU3ideal} have the following expressions
$$
\begin{array}{lcl} 
\omega &=& - \alpha_{12}  (a_1 + a_3) e^{12} - \alpha_{13}  (a_1 + a_5) e^{13} + \alpha_{14}  (a_3 + a_5) e^{14} + (- a_1 \alpha_{15}  - a_8 \alpha_{14} - \beta_1) e^{15} \\[2pt]
&& + (b \, \beta_1 - a_1 \alpha_{16} - a_4 \alpha_{12} - a_6 \alpha_{13}) e^{16} - \alpha_{23} (a_3 + a_5) e^{23} 
+ \alpha_{24}  (a_1 + a_5) e^{24} \\[2pt] && + (a_2 \alpha_{12} - a_3 \alpha_{25} - a_8 \alpha_{24}) e^{25} 
+ (- b \, a_2 \alpha_{12}  - b  \, \beta_2 - a_3 \alpha_{26} - a_6 \alpha_{23}  - \beta_2) e^{26}  \\[2pt]
 && +\alpha_{34} (a_1 + a_3) e^{34} + (a_2 \alpha_{13}- a_5 \alpha_{35}  - a_8 \alpha_{34}) e^{35}   \\[2pt]
&&  + (- b \, a_2 \alpha_{13} + a_4 \alpha_{23} - a_5  \alpha_{36} + \beta_3) e^{36}+ (a_1 \alpha_{45} + a_2 \alpha_{14} + a_3 \alpha_{45} + a_5 \alpha_{45}+ \beta_4) e^{45}   \\[2pt]
&& + (-b a_2 \alpha_{15} - a_2 \alpha_{16} + a_4 \alpha_{25} + a_6 \alpha_{35}  - a_8 \alpha_{46}) e^{56}  \\[2pt]
&& + (- b \, a_2 \alpha_{14}  + a_1 \alpha_{46}  + a_3 \alpha_{46}  + a_4 \alpha_{24} + a_5 \alpha_{46}  + a_6 \alpha_{34}) e^{46}  
\end{array} 
$$
 and
\[
\begin{split}
\psip =& \alpha_{12}  (e^{125} + e^{126} ) +\alpha_{13}  ( e^{135} -  (1 + b) e^{136})  - \alpha_{14}  b e^{146} + \alpha_{24}  (-e^{245}+ (1 + b) e^{246}) \\
	&+ \alpha_{23}  b e^{236}  -  \alpha_{34}  (e^{345} + e^{346}) + (- b \alpha_{15}  - \alpha_{16}) e^{156} +\alpha_{25}  (1 + b) e^{256}  - \alpha_{35} e^{356} + \alpha_{46}  e^{456}.
\end{split} 
\]
Assume that $\lambda(\psip) <0$, and let $J$ be the almost complex structure induced by $\psip$ and the orientation $e^{123456}$. 
We will show that there exists a nonzero vector $x\in\frs$ such that  $\omega (x, J x) = 0$. 
 To prove this, we consider  the nilpotent  derivation $S\in\Der(\frs)$ with associated matrix 
$$
S= \left( \begin{array}{cccccc} 0&0& 0 & 0&  s_1 & - b \, s_1\\[2pt]
 0 & 0 &  0 & 0 &  0 & s_2\\[2pt]
0 &  0 &  0 & 0 & 0 & s_3\\[2pt]
0 &  0 &  0 &  0 & s_4 &  0\\[2pt]
 0 & 0 &  0 &  0 &  0 &  0\\[2pt]
 0 & 0 &  0 &  0 &  0 &  0 
 \end{array}
 \right ),
$$
and  the automorphism  $F = \exp (S)$.   
With similar computations as in the proof of Proposition \ref{3434}, we see that there exist certain $s_i$, $1 \leq i \leq 4$, 
 such that $F^* \omega (e_5, e_6) =0$ and  $J_{F^*\psip} (e_5) \in \la e_5, e_6 \ra$, where $J_{F^*\psip} = F^{-1}\circ J\circ F$. This follows comparing
\begin{align*}
F^*\psip=& \alpha_{12}  (e^{125} + e^{126}) +  \alpha_{13} (e^{135} - (1+b) e^{136}) -  b \, \alpha_{14}   e^{146} +  b \, \alpha_{23}  e^{236} \\ 
+& \alpha_{24}  (-e^{245} +  (1+b) e^{246}) - \alpha_{34} (e^{345} + e^{346})\\
-& (b \, \alpha_{14}  s_4 +b \alpha_{15} + \alpha_{12}  s_2 + \alpha_{13}  s_3 + \alpha_{16}) e^{156} \\
+& (1+b)(-  \alpha_{12}  s_1 + \alpha_{24}  s_4 + \alpha_{25})e^{256} \\
+& (\alpha_{13}  s_1 - \alpha_{34}  s_4- \alpha_{35})e^{356} + (b \,   \alpha_{14} s_1 - \alpha_{24}  s_2-\alpha_{34}  s_3 + \alpha_{46})e^{456},
\end{align*}
with
\begin{align*} 
F^*\omega(e_5,e_6)=-& a_1 s_1 ( b \alpha_{15}  +\alpha_{12}  s_2 + \alpha_{13}  s_3  + \alpha_{16}) 
+ a_1 s_4 ( \alpha_{46}  -  \alpha_{34}  s_3 - \alpha_{24}  s_2)  \\
+& a_3 s_2 ( - \alpha_{12}  s_1 + \alpha_{25})
+ a_3 s_4  (b \alpha_{14}  s_1 -\alpha_{34}  s_3 + \alpha_{46}) \\
+& a_5 s_3 (- \alpha_{13}  s_1 +  \alpha_{35})
+ a_5 s_4 (b \alpha_{14}  s_1 - \alpha_{24}  s_2 + \alpha_{46}) \\
-& a_2 ( b \alpha_{14}  s_4 + b\alpha_{15}  + \alpha_{12}  s_2 +  \alpha_{13}  s_3 + b_5) 
+ a_4 (-  \alpha_{12}  s_1 + \alpha_{24}  s_4 + \alpha_{25}) \\
+& a_6 (- \alpha_{13}  s_1 + \alpha_{34}  s_4 + \alpha_{35})
+ a_8 ( - b \alpha_{14}  s_1 + \alpha_{24}  s_2 +  \alpha_{34}  s_3 -  \alpha_{46}),
\end{align*}
and observing that $F^*\omega(e_5,e_6)=0$ if the coefficients of $e^{156}$, $e^{256}$, $e^{356}$ and $e^{456}$ in the expression of $F^*\psip$ vanish. 
This last condition gives the linear system
$$
\left\{\begin{array} {l}
 b  \, \alpha_{14}  s_4 + b  \alpha_{15}  +  \alpha_{12}  s_2 +  \alpha_{13}  s_3  = - \alpha_{16}, \\
- \alpha_{12}  s_1 +   \alpha_{24}  s_4= -\alpha_{25}, \\
\alpha_{13}  s_1 - \alpha_{34} s_4 = \alpha_{35}, \\
b  \alpha_{14} s_1 -  \alpha_{24}  s_2 -  \alpha_{34}  s_3= - \alpha_{46},
\end{array}
\right.
$$
 which has a unique solution $(\bar{s}_1,\bar{s}_2, \bar{s}_3,\bar{s}_4)$ under the constraint $\lambda(\psip)<0$.  
 The choice $(\bar{s}_1,\bar{s}_2, \bar{s}_3,\bar{s}_4)$ gives
\begin{align*}
F^*\psip	&=\alpha_{12}  (e^{125} + e^{126}) + \alpha_{13} (e^{135} - (1+b)e^{136}) - \alpha_{14}  b e^{146} +  \alpha_{23}  b e^{236} \\ 
		&\quad+  \alpha_{24} (-e^{245} +  (1+b) e^{246}) -  \alpha_{34}(e^{345} + e^{346}),
\end{align*}
and $F^*\omega(e_5,e_6)=0$.
From the expression of $F^*\psip$, we observe that $J_{F^*\psip} e_5 \in \la e_5, e_6 \ra$. In detail, 
\[
\frac{1}{2}\,\sqrt{-\lambda(F^*\psip)}\,J_{F^* \psip}e_5= - ((b+1)  \alpha_{13}   \alpha_{24} +  \alpha_{12} \alpha_{34})\, e_5 + (\alpha_{12}   \alpha_{34} - \alpha_{13} \alpha_{24})\, e_6.
\]
We then have that $\omega(x,Jx)=0$ for $x = Fe_5$.

\smallskip

\smallskip

Let us now focus on the case $a \neq - 1$.
The nilradical of  $\frs$ is still the abelian ideal $\frn =\langle e_1,e_2,e_3, e_4\rangle$, 
and the generic derivation $D\in\Der(\frs)$ for which $\frg = \frs \rtimes_D\R$ is strongly unimodular 
must have the following matrix representation with respect to $\mathcal{B}$
\[
D = \left(
\begin{array}{cccccc}
a_{1} & 0 & 0 &0& a_2 & \frac{a_2 \,  b }{a}
\\[2pt]
 0 & a_{3} & 0 & 0 & a_4 & \frac{a_4 (1 + b)}{a + 1}
\\[2pt]
 0 & 0 & a_{5} & 0 & 0& a_6
\\[2pt]
 0 & 0 & 0 & -a_{5}-a_{3}-a_{1}  & a_8 & 0 
\\[2pt]
 0 & 0 & 0 & 0 &0&0
\\[2pt]
 0 & 0 & 0 & 0 & 0&  0
\end{array}\right),
\]
where $a_i\in\R$.

Let $$
\begin{array}{lcl} t_2  &\coloneqq &  -\frac{1}{a(1+a)}(a^2 b^2 \alpha_{13}^2 \alpha_{24}^2-2 a^2 b^2 \alpha_{13}   \alpha_{14}   \alpha_{23}   \alpha_{24}  + a^2 b^2 \alpha_{14}^2 \alpha_{23}^2-2  a^2 b \alpha_{12}  \alpha_{13} \alpha_{24}  \alpha_{34}\\
&&  -2  a^2 b  \alpha_{12}   \alpha_{14}   \alpha_{23}   \alpha_{34} +2  a^2 b  \alpha_{13}^2 \alpha_{24}^2  -2  a^2 b \alpha_{13}   \alpha_{14}   \alpha_{23}   \alpha_{24}  +2  a b^2  \alpha_{12}   \alpha_{13}    \alpha_{24}  \alpha_{34}\\[2pt]
&& +2  a b^2  \alpha_{12}    \alpha_{14}    \alpha_{23}    \alpha_{34} -2  a b^2  \alpha_{13}   \alpha_{14}    \alpha_{23}  \alpha_{24}  +2  a b^2  \alpha_{14}^2  \alpha_{23}^2+ a^2 \alpha_{12}^2 \alpha_{34}^2\\[2pt]
&& -2  a^2 \alpha_{12}  \alpha_{13}  \alpha_{24}   \alpha_{34} + 
 a^2  \alpha_{13}^2 \alpha_{24}^2-2  a b  \alpha_{12}^2 \alpha_{34}^2+2  a b \alpha_{12}   \alpha_{13}  \alpha_{24} \alpha_{34}-2  a b \alpha_{12}  \alpha_{14}   \alpha_{23}   \alpha_{34}\\[2pt]
 &&  -2  a b  \alpha_{13}    \alpha_{14}   \alpha_{23}  \alpha_{24} +b^2  \alpha_{12}^2  \alpha_{34}^2+2 b^2  \alpha_{12}   \alpha_{14}   \alpha_{23}  \alpha_{34}+b^2  \alpha_{14}^2  \alpha_{23}^2). 
\end{array}
$$
We will study  the cases $t_2 \neq 0$ and $t_2 =0$ separately.

If $t_2 \neq 0$, we claim that there exists a nonzero vector $x\in\frs$ such that $\omega(x,Jx)$=0. 
The discussion is similar to the previous case. 
Here, we consider  the nilpotent  derivation $S$ with matrix representation 
$$
S= \left( \begin{array}{cccccc} 0&0& 0 & 0&  s_1 & \frac{b s_1}{a}\\[2pt]
 0 & 0 &  0 & 0 &  s_2 & \frac{s_2 (1+b)}{1+a} \\[2pt]
0 &  0 &  0 & 0 & 0 & s_3\\[2pt]
0 &  0 &  0 &  0 & s_4 &  0\\[2pt]
 0 & 0 &  0 &  0 &  0 &  0\\[2pt]
 0 & 0 &  0 &  0 &  0 &  0 
 \end{array}
 \right )
$$
and the automorphism  $F = \exp (S)$. 
Requiring the coefficients of  $e^{156}$,  $e^{256}$, $e^{356}$, $e^{456}$ in the expression of $F^* \psip$ to be zero gives 
the following  linear system in the variables $s_i,$ $1 \leq i\leq 4$, 
$$
\left \{ 
\begin{array} {l}
 \alpha_{13} as_3- \alpha_{14} bs_4-\frac{1+b}{1+a} \alpha_{12} s_2+ \alpha_{12} s_2= -a \,  \alpha_{16}+b \, \alpha_{15}, \\
- \alpha_{23} s_3a+ \alpha_{24} s_4b+\frac{b}{a} \alpha_{12} s_1- \alpha_{12} s_1- \alpha_{23} s_3+ \alpha_{24} s_4=-(1+b) \alpha_{25} +(1+a) \alpha_{26}, \\
  \frac{1+b}{1+a} a \alpha_{23} s_2- \alpha_{23} bs_2+\frac{1+b}{1+a} \alpha_{23} s_2+ \alpha_{13} s_1 \alpha_{34} s_4=\, \alpha_{35}, \\
\frac{1+b}{1+a}  a  \alpha_{24} s_2 - \alpha_{24} s_2b-\frac{b}{a} \alpha_{14} s_1- \alpha_{24} s_2- \alpha_{34} s_3=- \alpha_{46}.
\end{array} 
\right. 
$$  
This system has a unique solution  $(\bar{s}_1,\bar{s}_2,\bar{s}_3,\bar{s}_4)$ under the constraint  $t_2 \neq 0$.
In such a case, we have  $F^*\omega(e_5,e_6)=0$ and $J_{F^*\psip}e_5\in\langle e_5,e_6\rangle$, whence the claim follows.

\smallskip

If $t_2=0$,  we claim that  $g = \omega(\cdot,J\cdot)$  cannot be definite. 
To prove this, suppose by contradiction that $g$ is positive  (or negative) definite, and consider $g' \coloneqq \sqrt{-\lambda(\psip)}\,g$. Then, 
\begin{align*}
g_{11}' &= 4 \alpha_{12}  \alpha_{13}   \alpha_{14}   ((a_1 + a_3 + a_5) a - b a_5 + a_1), \\
g'_{33} &= 4  \alpha_{13}   \alpha_{23}   \alpha_{34}   ((a_1 + a_3 + a_5) a - b a_5 + a_1),  \\
g'_{44} &= -4  \alpha_{14}  \alpha_{24}   \alpha_{34}  ((a_1 + a_3 + a_5) a - b a_5 + a_1). 
\end{align*}
Therefore, $ \alpha_{12}  \alpha_{13}  \alpha_{14} \alpha_{23}  \alpha_{24}  \alpha_{34} \neq 0$, and 
the polynomials $p_1\coloneqq \alpha_{12}  \alpha_{34}$, $p_2\coloneqq \alpha_{14} \alpha_{23}$ and $p_3\coloneqq - \alpha_{13} \alpha_{24}$ must have the same sign. 

Since $b(a+1) \alpha_{23} \neq 0$, the condition $t_2=0$ can be seen as a second order equation in the variable $\alpha_{14}$. 
We can solve it provided that $a(b+1) \alpha_{13}  \alpha_{24} (a-b)b_1 \alpha_{34}  \geq 0$, obtaining the solutions:
$$
\alpha_{14}^{\pm}= \frac{1}{\alpha_{23} b(a+1)}\left( a(b+1) \alpha_{13} \alpha_{24}  + (a-b) \alpha_{12}  \alpha_{34}  \pm 2 \sqrt{(a(b+1) \alpha_{13}  \alpha_{24} (a-b) \alpha_{12}  \alpha_{34})}\right).
$$
A case by case analysis ensures that the condition $a(b+1) \alpha_{13}  \alpha_{24} (a-b)b_1 \alpha_{34}  \geq 0$ is not compatible with the constraint on $p_1,p_2,p_3$. 
To check this, assume that $a(b+1) \alpha_{13}  \alpha_{24} (a-b) \alpha_{12} \alpha_{34}  \geq 0$ and that $p_1$ and $p_3$ have the same sign.  
Then, we can distinguish the  four cases:
\begin{enumerate}
\item[1)]  $ \alpha_{12} \alpha_{34} >0$, $\alpha_{13}  \alpha_{24} <0$, $a-b\leq 0$, $a(b+1)\geq 0$,
\item[2)]  $\alpha_{12} \alpha_{34} >0$, $\alpha_{13} \alpha_{24} <0$, $a-b \geq 0$, $a(b+1)\leq 0$,
\item[3)]  $\alpha_{12}  \alpha_{34} <0$, $\alpha_{13} \alpha_{24} >0$, $a-b \leq 0$, $a(b+1)\geq 0$,
\item[4)]  $\alpha_{12}\alpha_{34}<0$, $ \alpha_{13}  \alpha_{24} >0$, $a-b \geq 0$, $a(b+1)\leq 0$.
\end{enumerate}
Under the assumptions of case $1)$,
$$
\alpha_{14}^{\pm} \alpha_{23} = -\frac{1}{b(a+1)}\left( \sqrt{-a(b+1) \alpha_{13} \alpha_{24}} \mp \sqrt{-(a-b) \alpha_{12} \alpha_{34}} \right)^2.
$$
If in addition $p_1$ and $p_2$ have the same sign,  then 
$b(a+1)<0$. This condition is incompatible with the inequalities  $a\leq b$ and $a(b+1) \geq 0$, since these conditions imply $a\leq b$ and $-a \leq ab < -b$. 
In the remaining cases, we can proceed in a  similar way.
\end{proof}

\smallskip\noindent
{\bf Acknowledgements.}  
A.F.~and A.R.~were supported by GNSAGA of INdAM and by the project PRIN 2017  ``Real and Complex Manifolds: Topology, Geometry and Holomorphic Dynamics''. 
L.M.M.~acknowledges financial support by a FPU Grant (FPU16/03475) and its research stay program (EST19/00747).  
The authors would like to thank the anonymous referee for the valuable comments and suggestions that helped in improving the presentation of the results.

\appendix

\section{}\label{appendix}
In this appendix, we list  the structure equations of all six-dimensional unimodular decomposable solvable non-nilpotent Lie algebras that exist up to isomorphism. 
 {\small
\begin{table}[H]
\renewcommand\arraystretch{1.2}
\begin{tabular}{|c|c|}  
\hline
$\mathfrak{s}$&\small $(de^1,de^2,de^3,de^4,de^5,de^6)$\normalsize \\ \hline \hline
$\frg_{3,4}^{-1} \oplus \R^3$  &\small$(-e^{13},  e^{23}, 0,0,0,0)$ \normalsize \\ \hline
$\frg_{3,4}^{-1} \oplus \frg_{3,1}$  &\small$(-e^{13},  e^{23}, 0, -e^{56},0,0)$ \normalsize \\ \hline
$\frg_{3,4}^{-1} \oplus \frg_{3,4}^{-1}$  &\small$(-e^{13}, e^{23},0,-e^{46}, e^{56},0)$ \normalsize \\ \hline
$\frg_{3,4}^{-1} \oplus \frg_{3,5}^0$  &\small$(-e^{13}, e^{23}, 0,-e^{56},  e^{46},0)$ \normalsize \\ \hline
$\frg_{3,5}^0 \oplus \R^3$  &\small$(-e^{23},  e^{13}, 0, 0,0,0)$ \normalsize \\ \hline
$\frg_{3,5}^0 \oplus \frg_{3,1}$  &\small$(-e^{23},  e^{13}, 0,-e^{56},0,0)$ \normalsize \\ \hline
$\frg_{3,5}^0 \oplus \frg_{3,5}^0$  &\small$(-e^{23},  e^{13}, 0,-e^{56},  e^{46},0)$ \normalsize \\ \hline
$\frg_{4,2}^{–2} \oplus \R^2$  &\small$(-2 e^{14}, e^{24} + e^{34}, e^{34}, 0,0)$ \normalsize \\ \hline
$\frg_{4,5}^{p, - p -1} \oplus \R^2$  &\small$(-e^{14}, - p e^{24}, (p + 1) e^{34}, 0,0,0)$, $p \in [- \frac{1}{2},0)$ \normalsize \\ \hline
$\frg_{4,6}^{-2p,p} \oplus \R^2$  &\small$(2p e^{14}, - p e^{24} - e^{34}, e^{24} - p e^{34},0,0,0)$, $ p >0$ \normalsize \\ \hline
$\frg_{4,8}^{-1} \oplus \R^2$  &\small$(-e^{23}, -e^{24}, e^{34}, 0,0,0)$ \normalsize \\ \hline
$\frg_{4,9}^0 \oplus \R^2$  &\small$(- e^{23}, - e^{34}, e^{24}, 0,0,0))$ \normalsize \\ \hline
$\frg_{5,7}^{p,q,r} \oplus \R$  &\small$(-e^{15}, - p e^{25}, - q e^{35}, - r e^{45}, 0, 0)$,  $- 1 \leq r \leq q \leq p \leq 1,$ $pqr \neq 0,$ $p + q + r = -1$ \normalsize \\ \hline
$\frg_{5,8}^{-1} \oplus \R$  &\small$(- e^{25}, 0, - e^{35}, e^{45}, 0,0,0)$ \normalsize \\ \hline
$\frg_{5,9}^{p,-2 -p}\oplus \R$  &\small$(- e^{15} - e^{25}, - e^{25}, - p e^{35}, (2 + p) e^{45}, 0,0)$,  $ p \geq -1$ \normalsize \\ \hline
$\frg_{5,11}^{-3}\oplus \R$  &\small$(- e^{15} - e^{25}, - e^{25} - e^{35},  - e^{35}, 3 e^{45}, 0,0)$ \normalsize \\ \hline
$\frg_{5,13}^{-1 - 2q, q, r} \oplus \R$  &\small$(- e^{15}, (1 + 2 q) e^{25}, - q e^{35} -  r e^{45}, r e^{35} - q e^{45}, 0,0)$,  $q \in [-1,0],$ $q \neq - \frac 12,$ $r \neq 0$ \normalsize \\ \hline
$\frg_{5,14}^0 \oplus \R$  &\small$(- e^{25}, 0, - e^{45}, e^{35}, 0,0)$ \normalsize \\ \hline
$\frg_{5,15}^{-1} \oplus \R$  &\small$(- e^{15} - e^{25}, - e^{25}, e^{35} - e^{45}, e^{45}, 0,0)$ \normalsize \\ \hline
$\frg_{5,16}^{-1,q} \oplus \R $  &\small$(- e^{15} - e^{25}, - e^{25}, e^{35} - q e^{45}, qe^{35} + e^{45}, 0,0)$,  $q \neq 0$ \normalsize \\ \hline
$\frg_{5,17}^{p,-p,r} \oplus \R$  &\small$(- p e^{15} - e^{25}, e^{15} - p e^{25}, p e^{35} - r e^{45}, r e^{35} + p e^{45}, 0,0)$,  $r \neq 0$ \normalsize \\ \hline
$\frg_{5,18}^0 \oplus \R$  &\small$(- e^{25} - e^{35}, e^{15} - e^{45}, - e^{45}, e^{35}, 0,0)$ \normalsize \\ \hline
$\frg_{5,19}^{p,-2p - 2} \oplus \R $  &\small$(- e^{23} - (1 + p) e^{15}, - e^{25}, - p e^{35}, (2 p + 2) e^{45}, 0,0)$, $p \neq - 1$ \normalsize \\ \hline
$\frg_{5,20}^{-1}\oplus \R$  &\small$(- e^{23} - e^{45}, - e^{25}, e^{35}, 0,0,0)$   \normalsize \\ \hline
$\frg_{5,23}^{-4}\oplus \R$  &\small$(- e^{23} - 2 e^{15}, - e^{25}, - e^{25} - e^{35}, 4 e^{45}, 0,0)$   \normalsize \\ \hline
$\frg_{5,25}^{4,4p}\oplus \R$  &\small$(- e^{23} - 2 p e^{15}, - p e^{25} + e^{35}, - e^{25} - p e^{35}, 4 p e^{45}, 0,0)$, $p \neq 0$  \normalsize \\ \hline
$\frg_{5,26}^{0,\varepsilon} \oplus \R$  &\small$(- e^{23} - \varepsilon e^{45}, e^{35}, - e^{25}, 0,0,0),$    $\varepsilon = \pm 1$ \normalsize \\ \hline
$\frg_{5,28}^{- 3 / 2} \oplus \R$  &\small$(- e^{23} + \frac 12 e^{15}, \frac 32 e^{25}, - e^{35}, - e^{35} - e^{45},  0,0)$ \  \normalsize \\ \hline
$\frg_{5,30}^{-  4 / 3} \oplus \R$  &\small$( - e^{24} - \frac 23 e^{15}, - e^{34} + \frac 13 e^{25}, \frac 43 e^{35}, - e^{45}, 0,0)$ \  \normalsize \\ \hline
$\frg_{5,33}^{-1,-1}\oplus \R$  &\small$(- e^{14}, - e^{25}, e^{34} + e^{35}, 0,0,0)$ \  \normalsize \\ \hline
$\frg_{5,35}^{-2,0}\oplus \R$  &\small$(2 e^{14}, - e^{24} - e^{35}, e^{25}- e^{34}, 0,0,0)$ \  \normalsize \\ \hline
\end{tabular}  
\caption{Isomorphism classes of six-dimensional unimodular decomposable solvable non-nilpotent Lie algebras.}
\label{tabledec}
\renewcommand\arraystretch{1}
\end{table} 
}

\begin{remark}
As for the Lie algebras appearing in Table \ref{tabledec} and depending on some parameters, 
with the exception of $\frg_{5,9}^{p,-2 -p}$, all of the corresponding simply connected Lie groups admit a lattice for certain values of the parameters, see \cite{Boc}. 
\end{remark}

\end{document}